\documentclass[12pt]{amsart}

\usepackage{amsmath,bm,amssymb,amsthm,textcomp}
\usepackage{enumitem}
\usepackage{mathtools}
\usepackage{hyperref}
\usepackage[numbers,sort&compress]{natbib}
\usepackage{tikz-cd}
\usetikzlibrary{cd}
\DeclareMathOperator{\hdim}{\dim_H}
\DeclareMathOperator{\realizable}{re}

\theoremstyle{plain}
\newtheorem{theorem}{Theorem}[section]
\newtheorem{lemma}[theorem]{Lemma}
\newtheorem{proposition}[theorem]{Proposition}

\theoremstyle{definition}
\newtheorem{example}[theorem]{Example}

\theoremstyle{remark}
\newtheorem{remark}[theorem]{Remark}

\numberwithin{equation}{section}

\allowdisplaybreaks
\makeatletter
\newcount\es@N \newcount\es@q \newcount\es@a
\newcount\es@num \newcount\es@den \newcount\es@d \newcount\es@r \newcount\es@t
\newcount\es@pm \newcount\es@qm \newcount\es@p \newcount\es@qq
\newcount\es@sP \newcount\es@yE \newcount\es@W \newcount\es@H \newcount\es@off
\newcount\es@dot \newcount\es@half
\def\es@euclid{%
  \ifnum\es@num>\z@
    \es@d=\numexpr\es@den/\es@num\relax
    \ifnum\numexpr\es@d*\es@num\relax>\es@den \advance\es@d\m@ne \fi
    \es@r=\numexpr\es@den-\es@d*\es@num\relax
    \es@den=\es@num \es@num=\es@r
    \es@t=\numexpr\es@d*\es@p+\es@pm\relax \es@pm=\es@p \es@p=\es@t
    \es@t=\numexpr\es@d*\es@qq+\es@qm\relax \es@qm=\es@qq \es@qq=\es@t
    \es@t=\numexpr\es@qq*\es@a-\es@p*\es@q\relax
    \advance\es@sP\es@t
    \advance\es@yE\numexpr\es@t*\es@H/(\es@q*\es@qq)\relax
    \expandafter\es@euclid
  \fi}
\def\es@point{%
  \es@num=\es@a \es@den=\es@q
  \es@pm=\@ne \es@qm=\z@ \es@p=\z@ \es@qq=\@ne
  \es@sP=\es@a
  \es@yE=\numexpr\es@a*\es@H/\es@q\relax
  \es@euclid
  \ifnum\es@den=\@ne 
    \es@t=\numexpr\es@a*\es@W/\es@q-\es@half\relax
    \pgfsys@rect{\the\es@t sp}{\the\numexpr\es@yE-\es@half\relax sp}{\the\es@dot sp}{\the\es@dot sp}%
    \pgfsys@rect{\the\numexpr\es@t+\es@off\relax sp}{\the\numexpr\es@sP*\es@H/\es@q-\es@half\relax sp}{\the\es@dot sp}{\the\es@dot sp}%
  \fi}
\def\es@loopa{%
  \ifnum\es@a<\es@q
    \es@point \advance\es@a\@ne
    \expandafter\es@loopa
  \fi}
\def\es@loopq{%
  \ifnum\es@q>\es@N\else
    \es@a=\z@ \es@loopa \advance\es@q\@ne
    \expandafter\es@loopq
  \fi}
\newcommand{\errsumdots}[5]{%
  \es@N=#1\relax \es@W=\dimexpr#2\relax \es@H=\dimexpr#3\relax
  \es@off=\dimexpr#4\relax \es@dot=\dimexpr#5\relax \es@half=\numexpr\es@dot/2\relax
  \es@q=\@ne \es@loopq \pgfsys@fill}
\makeatother

\begin{document}

\title[Graph of the unweighted error-sum function of continued fractions]{Hausdorff dimension of the graph of the unweighted error-sum function of continued fractions}

\author{Min Woong Ahn}
\address{Department of Mathematics Education, Dankook University, 152, Jukjeon-ro, Suji-gu, Yongin-si, Gyeonggi-do, 16890, Republic of Korea}
\email{mwahn@dankook.ac.kr}

\date{\today}

\subjclass[2020]{Primary 11A55, 11K50; Secondary 26A18, 28A80, 33E20, 37E05}
\keywords{Continued fraction, error-sum function, Hausdorff dimension, graph of a function, symbolic space}

\begin{abstract}
We study the unweighted error-sum function $\mathcal{E}(x) \coloneqq \sum_{n \geq 0} (x - p_n(x)/q_n(x))$, where $p_n(x)/q_n(x)$ denotes the $n$th convergent of the continued fraction expansion of $x \in \mathbb{R}$, and we prove that the graph of $\mathcal{E}$ has Hausdorff dimension $1$. The key ingredient is an exact formula for the oscillation of $\mathcal{E}$ on fundamental intervals, which we obtain in a symbolic space of continued fraction sequences containing non-realizable sequences. Using the same framework, we also give a new derivation of the upper bound $3/2$ of Ridley and Petruska (2000) for the Hausdorff dimension of the graph of the relative error-sum function $P(x) \coloneqq \sum_{n \geq 0} (q_n(x)x - p_n(x))$. Moreover, we present an alternative argument based on M\"obius inversion and summation over coprime convergent denominators, and we explain why the square coverings used for $P$ cannot yield an upper bound smaller than $3/2$.
\end{abstract}

\maketitle

\tableofcontents

\section{Introduction} \label{introduction}

Given $x \in \mathbb{R}$, we denote the fractional part of $x$ by $\{ x \} \coloneqq x - \lfloor x \rfloor \in [0,1)$, and set
\begin{align} \label{CF algorithm 1}
d_1 (x) \coloneqq \begin{cases} \lfloor 1/\{x\} \rfloor, &\text{if } \{ x \} \neq 0; \\ \infty, &\text{if } \{ x \} = 0, \end{cases}
\quad \text{and} \quad
T (x) \coloneqq \begin{cases} 1/\{x\} - d_1(x), &\text{if } \{ x \} \neq 0; \\ 0, &\text{if } \{ x \} = 0. \end{cases}
\end{align}
The map $T \colon \mathbb{R} \to [0,1)$ is called the {\em continued fraction (CF) transformation}. For each $n \in \mathbb{N}$, we define
\begin{align} \label{CF algorithm 2}
d_n (x) \coloneqq d_1 (T^{n-1}(x)),
\end{align}
where $d_n(x)$ is called the $n$th {\em digit} of the {\em CF expansion} of $x$. Setting $d_0(x) \coloneqq \lfloor x \rfloor$, the algorithms \eqref{CF algorithm 1} and \eqref{CF algorithm 2} yield the unique expansion $x = [d_0(x); d_1(x), d_2(x), \dotsc]$. Here, for any sequence $(\sigma_0, \sigma_1, \sigma_2, \dotsc) \in \mathbb{Z} \times (\mathbb{N} \cup \{ \infty \})^{\mathbb{N}}$, we define
\begin{align*}
[\sigma_0; \sigma_1, \sigma_2, \sigma_3, \dotsc] \coloneqq \sigma_0 + \cfrac{1}{\sigma_1 + \cfrac{1}{\sigma_2 + \cfrac{1}{\sigma_3 + \ddots}}},
\end{align*}
with the conventions $1/\infty = 0$, $\infty + \infty = \infty$, and $\infty + c = \infty$ for any $c \in \mathbb{R}$. If $x = [d_0(x); d_1(x), \dotsc, d_n(x), \infty, \infty, \dotsc]$ with $d_n(x) < \infty$, we say that $x$ has a {\em finite} CF expansion of {\em length $n$} and write $x = [d_0(x); d_1(x), \dotsc, d_n(x)]$. In this case, $d_n(x) > 1$ whenever $n \geq 1$ (see Proposition~\ref{CF realizable sequence} below).

We remark that the above construction ensures the uniqueness of the CF expansion. Without it, every rational number in $(0,1)$ would admit two distinct finite continued fraction expansions in the classical sense. Indeed, if $x$ has a finite CF expansion of length $n$, then its last digit satisfies $d_n(x) = (d_n(x)-1) + \frac{1}{1}$, with $d_n(x)-1 \in \mathbb{N}$, which produces an alternative longer expansion. Thus, the term ``CF expansion'' throughout this paper refers specifically to the expansion generated by algorithms \eqref{CF algorithm 1} and \eqref{CF algorithm 2}.

In 2000, Ridley and Petruska \cite{RP00} studied the {\em error-sum function} for CF expansions $P \colon \mathbb{R} \to \mathbb{R}$, defined by
\[
P( x ) \coloneqq \sum_{n \geq 0} (q_n(x)x-p_n(x)),
\]
where $( p_n(x) / q_n(x) )_{n \geq 0}$ denotes the sequence of {\em convergents} to $x \in \mathbb{R}$, given by
\[
\frac{p_n(x)}{q_n(x)} \coloneqq [d_0(x); d_1(x), d_2(x), \dotsc, d_n(x), \infty, \infty, \dotsc],
\]
with initial values $p_0(x) \coloneqq d_0(x)$, $q_0(x) \coloneqq 1$, and $\gcd (p_n(x), q_n(x)) = 1$ for all $n \geq 1$. They referred to $P$ as an error-sum function because each term $q_n(x)x-p_n(x)$ measures the approximation error $x-p_n(x)/q_n(x)$ of the $n$th convergent, scaled relative to the spacing between adjacent rationals with denominator $q_n(x)$. Among other properties, they investigated the continuity structure of $P$ and the Hausdorff dimension of its graph. They also showed that the graph of $P$ exhibits a fractal-like appearance, although it lacks strict self-similarity {(see Figure~\ref{fig:graphs})}. Despite the passage of more than two decades, the exact Hausdorff dimension of the graph has not been determined, and their original upper bound has not been improved. Thus, determining the Hausdorff dimension of the graph of $P$ remains
a natural open problem at the intersection of number theory and fractal geometry. For further developments on error-sum functions in CF expansions, see, e.g., \cite{Els11, Els14} by Elsner and \cite{ES11, ES12} by Elsner and Stein.

Inspired by Shen and Wu's work \cite{SW07} on the unweighted error-sum function of the L{\"u}roth series and our previous study \cite{Ahn23} on the Pierce expansion, we consider the analogous unweighted error-sum function associated with the CF expansion. This function $\mathcal{E} \colon \mathbb{R} \to \mathbb{R}$ is defined by
\[
\mathcal{E} (x) \coloneqq \sum_{n \geq 0} \left( x - \frac{p_n(x)}{q_n(x)} \right)
\]
for each $x \in \mathbb{R}$. To distinguish the two functions studied in this paper, we henceforth refer to $P$ as the {\em relative} error-sum function and $\mathcal{E}$ as the {\em unweighted} error-sum function.

\begin{figure}[t]
\centering
\begin{tikzpicture}[font=\footnotesize]
\def\S{5.6cm}\def\D{7.5cm}
\foreach \ox/\lab in {0cm/{$\mathcal{E}(x)$}, \D/{$P(x)$}} {
  \begin{scope}[xshift=\ox]
    \draw ({-0.02*\S},{-0.02*\S}) rectangle ({1.02*\S},{1.02*\S});
    \foreach \t in {0,0.2,0.4,0.6,0.8,1} {
      \draw ({\t*\S},{-0.02*\S}) -- ++(0,3pt) ({\t*\S},{1.02*\S}) -- ++(0,-3pt);
      \draw ({-0.02*\S},{\t*\S}) -- ++(3pt,0) ({1.02*\S},{\t*\S}) -- ++(-3pt,0);
      \node[anchor=north] at ({\t*\S},{-0.02*\S-1pt}) {$\t$};
      \node[anchor=east] at ({-0.02*\S-1pt},{\t*\S}) {$\t$};
    }
    \node[anchor=north] at ({0.5*\S},{-0.02*\S-13pt}) {\normalsize $x$};
    \node[anchor=south, rotate=90] at ({-0.02*\S-20pt},{0.5*\S}) {\normalsize \lab};
  \end{scope}
}
\errsumdots{300}{\S}{\S}{\D}{0.4pt}
\end{tikzpicture}
\caption{{The graphs of the unweighted error-sum function $\mathcal{E}$ (left) and the relative error-sum function $P$ (right) on $[0,1)$, plotted at all rational points $p/q \in [0,1)$ with $q \leq 300$. For comparison, Ridley and Petruska \cite{RP00} plotted $P$ on $(-1,1)$ at rational points with denominators at most $60$.}}
\label{fig:graphs}
\end{figure}

{
Error-sum functions have also been defined and studied for various other representations of real numbers. For example, see \cite{DT11} for the tent map base series, \cite{SW07} for the L\"uroth series, \cite{SMZ06} for the alternating L\"uroth series, \cite{CWY14} for the $\alpha$-L\"uroth series, \cite{JS12} for the alternating Sylvester series, and our previous work \cite{Ahn23} for the Pierce expansion. In particular, the graph of the error-sum function was shown to have Hausdorff dimension $1$ in \cite{Ahn23, CWY14, JS12, SMZ06, SW07}, where, in the case of the $\alpha$-L\"uroth series, the partition $\alpha$ is assumed to have the eventually decreasing property \cite{CWY14}. For the L\"uroth-type series, the length of a fundamental interval of order $n$ is a product of factors each of which depends on a single digit, and this product structure facilitates the estimates of the covering sums. For continued fractions, however, the length of a fundamental interval depends on the digits only through the convergent denominators $q_n$ and $q_{n-1}$, which are determined by a recurrence relation. Moreover, every rational number in $(0,1)$ admits two symbolic representations of different lengths (see Proposition~\ref{inverse image of phi} below). These two features are the main sources of difficulty in the present setting.
}

In this paper, we determine the Hausdorff dimension of the graph of the unweighted error-sum function $\mathcal{E}$, and we provide a new derivation of a known upper bound for the graph of the relative error-sum function $P$. The main results are as follows.

\begin{theorem} \label{Hausdorff dimension of graph of E}
The graph of the unweighted error-sum function $\mathcal{E} \colon \mathbb{R} \to \mathbb{R}$ has Hausdorff dimension $1$.
\end{theorem}

\begin{theorem}[{\cite[Theorem~3.2]{RP00}}] \label{Hausdorff dimension of graph of P}
The graph of the relative error-sum function $P \colon \mathbb{R} \to \mathbb{R}$ has Hausdorff dimension at most $3/2$.
\end{theorem}

{
Theorem~\ref{Hausdorff dimension of graph of E} is the main result of this paper. On the other hand, Theorem~\ref{Hausdorff dimension of graph of P} is due to Ridley and Petruska \cite{RP00}, and we do not improve their upper bound; instead, we provide a new derivation of it within the framework developed for $\mathcal{E}$.

The key ingredient in the proof of Theorem~\ref{Hausdorff dimension of graph of E} is an exact formula for the oscillation of $\mathcal{E}$ on fundamental intervals. More precisely, for any $n \in \mathbb{N}$ and any digit block $\sigma$ of length $n$, the oscillation of $\mathcal{E}$ on the corresponding fundamental interval $I_\sigma$ equals $n+1$ times the length of $I_\sigma$ (Lemma~\ref{sup of difference of E in I sigma}). Combining this formula with the Fibonacci lower bound for the convergent denominators, we obtain a short covering argument (Section~\ref{sec:graphE}). The same framework yields the exact oscillation of $P$ on fundamental intervals, from which Theorem~\ref{Hausdorff dimension of graph of P} follows (Section~\ref{sec:graphP}). In Section~\ref{sec:alternative}, we present an alternative argument, in which we convert the relevant covering sums into sums over coprime pairs of convergent denominators and evaluate them by means of M\"obius inversion. This argument is not needed for the proofs of the main results, but it provides exact formulas for the covering series. Together with Remark~\ref{remark on 3/2}, it also explains why the square coverings used for $P$ in this paper cannot yield an upper bound smaller than $3/2$ for the Hausdorff dimension of the graph of $P$.

We remark that the statement of Theorem~\ref{Hausdorff dimension of graph of E} also appears in \cite[Theorem~5]{CLS25}. However, the proof given there is incomplete. The bound used there for the oscillation of $\mathcal{E}$ on fundamental intervals is not justified, and in the final step, the sum of the lengths of $I_{(a_1, \dotsc, a_i)}$ over all digit blocks $(a_1, \dotsc, a_n) \in \mathbb{N}^n$ is taken to be $1$ for every $i \leq n$, although it is infinite for $i < n$. We give a complete proof using the exact oscillation formula in Lemma~\ref{sup of difference of E in I sigma}.

Our strategy follows our previous work \cite{Ahn23} on Pierce expansions. Namely, we lift the error-sum function to a symbolic space containing finite, possibly non-realizable, digit sequences, prove the continuity of the lifted function, compute its oscillation on cylinder sets, and cover the graph by columns over fundamental intervals. The symbolic space of CF sequences and its topology were introduced in \cite{Ahn25}. However, the following ingredients are specific to continued fractions. Firstly, every rational number in $(0,1)$ has two symbolic representations, one of which is non-realizable. Consequently, the fundamental interval associated with a non-realizable sequence is open, and the supremum or infimum of $\mathcal{E}$ over a fundamental interval is not attained in general. The non-realizable sequences in the symbolic space provide the missing endpoint values (Lemmas~\ref{f image of fundamental interval} and \ref{phi of cylinder}), which allows us to compute the oscillations exactly (Lemmas~\ref{bounds for E star in Upsilon sigma} and \ref{sup of difference of E in I sigma}). Secondly, for $P$, the weights $q_k(x)$ are unbounded, so that the continuity of the lifted function at finite sequences and the monotonicity of the scaled approximation errors require separate arguments (Lemmas~\ref{scaled errors} and \ref{continuity of symbolic error sums}). Thirdly, the alternative argument in Section~\ref{sec:alternative} relies on the correspondence between finite CF sequences and pairs of coprime convergent denominators (Lemma~\ref{pair correspondence}).
}

We denote by $\mathbb{N}$ the set of positive integers, by $\mathbb{N}_0$ the set of non-negative integers, and by $\mathbb{N}_\infty \coloneqq \mathbb{N} \cup \{ \infty \}$ the set of extended positive integers. Let $\mathbb{I} \coloneqq [0,1) \setminus \mathbb{Q}$ denote the set of irrational numbers in the interval $[0,1)$. The Fibonacci sequence $(F_n)_{n \in \mathbb{N}_0}$ is defined recursively by $F_0 \coloneqq 0$, $F_1 \coloneqq 1$, and $F_{n+2} \coloneqq F_{n+1} + F_n$ for all $n \in \mathbb{N}_0$. {Denote by $\gamma \coloneqq (1+\sqrt{5})/2$ the golden ratio. By Binet's formula, or by induction on $n$, we have
\begin{align} \label{Fibonacci growth}
F_n \geq \gamma^{n-2} \quad \text{for each } n \in \mathbb{N}.
\end{align}}%
We denote by $\mathcal{L}$ the Lebesgue measure on $\mathbb{R}$, by $\mathcal{H}^s$ the $s$-dimensional Hausdorff measure, and by $\hdim$ the Hausdorff dimension.

{The paper is organized as follows. In Section~\ref{Preliminaries}, we recall some properties of continued fractions and of the symbolic space of CF sequences. In Section~\ref{Auxiliary results}, we determine the oscillation of $\mathcal{E}$ and $P$ on fundamental intervals. We prove Theorems~\ref{Hausdorff dimension of graph of E} and \ref{Hausdorff dimension of graph of P} in Sections~\ref{sec:graphE} and \ref{sec:graphP}, respectively. In Section~\ref{sec:alternative}, we present the alternative argument via coprime summation and M\"obius inversion. Finally, Section~\ref{sec:conclusions} contains some concluding remarks.}

\section{Preliminaries} \label{Preliminaries}

In this section, we recall some basic properties of continued fractions. Arithmetic properties can be found in the classical references \cite{HW08, IK02, Khi97}, while topological properties of the CF digit coding map $x \mapsto (d_1(x), d_2(x), \dotsc)$ can be found in \cite{Ahn25}.

{
The relative error-sum function $P$ is periodic with period $1$ (see \cite[Proposition~1.2]{RP00}). Likewise, the unweighted error-sum function $\mathcal{E}$ is also $1$-periodic. We give a proof which works for both functions.

\begin{proposition} \label{E is periodic}
For any $x \in \mathbb{R}$, we have $\mathcal{E}(x+1) = \mathcal{E}(x)$ and $P(x+1) = P(x)$.
\end{proposition}

\begin{proof}
Let $x \in \mathbb{R}$. Since the CF expansions of $x$ and $\{x\}$ differ only in their integer parts, we have $p_n(x)/q_n(x) = \lfloor x \rfloor + p_n(\{x\})/q_n(\{x\})$ for each $n \geq 0$. Here, both fractions are in lowest terms, and hence $q_n(x) = q_n(\{x\})$ and $p_n(x) = \lfloor x \rfloor q_n(x) + p_n(\{x\})$ for each $n \geq 0$. It follows that
\[
x - \frac{p_n(x)}{q_n(x)} = \{x\} - \frac{p_n(\{x\})}{q_n(\{x\})}
\quad \text{and} \quad
q_n(x) x - p_n(x) = q_n(\{x\}) \{x\} - p_n(\{x\})
\]
for each $n \geq 0$. Summing over $n \geq 0$, we obtain $\mathcal{E}(x) = \mathcal{E}(\{x\})$ and $P(x) = P(\{x\})$. Since $\{x+1\}=\{x\}$, the proposition follows.
\end{proof}
}

In view of Proposition \ref{E is periodic}, we henceforth restrict our attention to the half-open unit interval $[0,1)$. Moreover, since $d_0(x) = 0$ for all $x \in [0,1)$, we shall write $[d_1(x), d_2(x), \dotsc]$ instead of $[0; d_1(x), d_2(x), \dotsc]$ for convenience.

Define the {\em CF digit coding map} $f \colon [0,1) \to \mathbb{N}_\infty^\mathbb{N}$ by
\[
f(x) \coloneqq (d_k(x))_{k \in \mathbb{N}} = (d_1(x), d_2(x), d_3(x), \dotsc)
\]
for each $x \in [0,1)$.

We now introduce the symbolic space of CF digit sequences. This symbolic space, denoted by $\Sigma$, was extensively studied in \cite{Ahn25} in connection with the topological properties of the map $f$. Let $\Sigma_0 \coloneqq \{ \infty \}^\mathbb{N}$, and for each $n \in \mathbb{N}$, define
\[
\Sigma_n \coloneqq \mathbb{N}^n \times \{ \infty \}^{\mathbb{N} \setminus \{ 1, \dotsc, n \}}.
\]
For notational convenience, we will occasionally write an element $(\sigma_k)_{k \in \mathbb{N}} \in \Sigma_n$ as the finite sequence $(\sigma_1, \dotsc, \sigma_n)$, with the understanding that it denotes the infinite sequence $(\sigma_1, \dotsc, \sigma_n, \allowbreak \infty, \infty, \dotsc)$. Define
\[
\Sigma_\infty \coloneqq \mathbb{N}^\mathbb{N},
\]
and set
\[
\Sigma \coloneqq \bigcup_{n \in \mathbb{N}_0} \Sigma_n \cup \Sigma_\infty.
\]
We refer to each $\sigma \in \Sigma$ as a {\em CF sequence}. A sequence $\sigma \coloneqq (\sigma_k)_{k \in \mathbb{N}} \in \Sigma$ is called {\em realizable} if there exists $x \in [0,1)$ such that $d_k(x) = \sigma_k$ for all $k \in \mathbb{N}$. Let $\Sigma_{\realizable} \subseteq \Sigma$ denote the set of all realizable CF sequences.

{
We remark that the enlarged space $\Sigma$, rather than the set $\Sigma_{\realizable}$ of digit sequences of numbers in $[0,1)$, is essential for our purposes. In addition to the realizable sequences, the space $\Sigma$ contains non-realizable finite sequences such as $(\sigma_1, \dotsc, \sigma_{n-1}, 1)$. Such sequences represent the endpoints of fundamental intervals which are not attained in the intervals themselves (see Proposition~\ref{I sigma} and Example~\ref{endpoint example} below). They supply the missing endpoint values which are needed to compute the exact oscillation of the error-sum functions on fundamental intervals in Section~\ref{Auxiliary results}.
}

For each $\sigma \coloneqq (\sigma_k)_{k \in \mathbb{N}} \in \Sigma$ and $n \in \mathbb{N}$, define
\[
\sigma^{(n)} \coloneqq (\tau_k)_{k \in \mathbb{N}}\in \Sigma
\quad \text{by} \quad
\tau_k
\coloneqq
\begin{cases}
\sigma_k, &\text{if } 1 \leq k \leq n; \\
\infty, &\text{if } k > n.
\end{cases}
\]
Thus, $\sigma^{(n)} = (\sigma_1, \dotsc, \sigma_n, \infty, \infty, \dots)$. For notational consistency, we also define $\sigma^{(0)} \coloneqq (\infty, \infty, \dotsc) \in \Sigma_0$.

Let $\sigma \coloneqq (\sigma_k)_{k \in \mathbb{N}} \in \Sigma$. Define $\varphi_0 (\sigma) \coloneqq 0$, and for each $k \in \mathbb{N}$, define
\[
\varphi_k (\sigma) \coloneqq [\sigma_1, \sigma_2, \dotsc, \sigma_k, \infty, \infty, \dotsc].
\]
For each $\sigma \in \bigcup_{n \in \mathbb{N}_0} \Sigma_n$, define
\[
\varphi(\sigma) \coloneqq \lim_{k \to \infty} \varphi_k(\sigma),
\]
which exists because the sequence is eventually constant.

\begin{proposition} [{\cite[Proposition~1.1.2]{IK02}}] \label{definition of phi n}
Let $\sigma \in \Sigma_\infty$. Then the limit $\varphi (\sigma) \coloneqq \lim_{k \to \infty} \varphi_k (\sigma)$ exists and satisfies $\varphi (\sigma) \in (0,1) \setminus \mathbb{Q}$. Moreover, the CF expansion digits of $\varphi (\sigma)$ are precisely given by $\sigma$, i.e., $\sigma = f(\varphi (\sigma))$.
\end{proposition}

Thus, $\varphi$ is well defined on $\Sigma$, and hence determines a map $\varphi \colon \Sigma \to [0,1]$.

For each $\sigma \in \Sigma$ and $n \in \mathbb{N}$, define
\[
\frac{p_n^* (\sigma)}{q_n^* (\sigma)} \coloneqq \varphi_n (\sigma),
\]
where $\gcd(p_n^*(\sigma), q_n^*(\sigma)) = 1$. We extend this notation by defining $p_0^*(\sigma) \coloneqq 0$, $q_0^*(\sigma) \coloneqq 1$, $p_{-1}^*(\sigma) \coloneqq 1$, and $q_{-1}^*(\sigma) \coloneqq 0$. {Note that $p_n^*(\sigma)$ and $q_n^*(\sigma)$ depend only on the first $n$ terms $\sigma_1, \dotsc, \sigma_n$ of $\sigma$.}

We now summarize several standard properties of continued fraction convergents, which will be repeatedly used throughout the paper.

\begin{proposition} [See \cite{Ahn25, HW08, IK02, Khi97, RP00}] \label{metric properties}
For any $\sigma \coloneqq (\sigma_k)_{k \in \mathbb{N}} \in \Sigma$, the following hold.
\begin{enumerate}[label = \upshape(\roman*), ref = \roman*, leftmargin=*, widest=iii]
\item \label{metric properties(i)}
If $\sigma \in \Sigma_\infty$, then
\begin{align*}
p_k^* (\sigma) = \sigma_k p_{k-1}^* (\sigma) + p_{k-2}^* (\sigma), \quad q_k^* (\sigma) = \sigma_k q_{k-1}^* (\sigma) + q_{k-2}^* (\sigma)
\end{align*}
for each $k \in \mathbb{N}$.

If $\sigma \in \Sigma_n$ for some $n \in \mathbb{N}$, then the same equalities hold for all $1 \leq k \leq n$, and
\[
p_k^* (\sigma) = p_n^* (\sigma), \quad q_k^* (\sigma) = q_n^* (\sigma)
\]
for all $k \geq n$. {In particular, $1 = q_0^*(\sigma) \leq q_1^*(\sigma) \leq q_2^*(\sigma) \leq \dotsb$ for every $\sigma \in \Sigma$.}

\item \label{metric properties(ii)}
Suppose $\sigma \coloneqq (\sigma_1, \dotsc, \sigma_{n-1}, \sigma_n) \in \Sigma_n$ for some $n \in \mathbb{N}$ with $\sigma_n > 1$, and define
\[
\sigma' \coloneqq (\sigma_1, \dotsc, \sigma_{n-1}, \sigma_n-1, 1) \in \Sigma_{n+1}.
\]
Then
\[
\left\{
\begin{array}{r@{\;}c@{\;}l@{\quad}r@{\;}c@{\;}l@{\quad}l}
p_k^*(\sigma') &=& p_k^*(\sigma),
&
q_k^*(\sigma') &=& q_k^*(\sigma),
&
\text{for } 1 \le k \le n-1;
\\[0.2em]
p_n^*(\sigma') &=& p_n^*(\sigma)-p_{n-1}^*(\sigma),
&
q_n^*(\sigma') &=& q_n^*(\sigma)-q_{n-1}^*(\sigma);
&
\\[0.2em]
p_k^*(\sigma') &=& p_n^*(\sigma),
&
q_k^*(\sigma') &=& q_n^*(\sigma),
&
\text{for } k \ge n+1.
\end{array}
\right.
\]

\item \label{metric properties(iii)}
If $\sigma \in \Sigma_\infty$, then $q_k^* (\sigma) \geq F_{k+1}$ for each $k \in \mathbb{N}_0$.

If $\sigma \in \Sigma_n$ for some $n \in \mathbb{N}$, then the same inequality holds for all $0 \leq k \leq n$.

\item \label{metric properties(iv)}
If $\sigma \in \Sigma_\infty$, then
\[
p_{k+1}^*(\sigma) q_k^*(\sigma) - p_k^*(\sigma) q_{k+1}^*(\sigma) = (-1)^k
\]
for each $k \geq -1$.

If $\sigma \in \Sigma_n$ for some $n \in \mathbb{N}$, then the same identity holds for all $-1 \leq k \leq n-1$.

\item \label{metric properties(v)}
If $\sigma \in \Sigma_\infty$, then
\[
\frac{1}{q_k^*(\sigma) (q_{k+1}^*(\sigma) + q_k^*(\sigma))} < (-1)^k ( \varphi (\sigma) - \varphi_k (\sigma) ) < \frac{1}{q_k^* (\sigma) q_{k+1}^*(\sigma)}
\]
for each $k \in \mathbb{N}_0$.

{If $\sigma \in \Sigma_n$ for some $n \geq 2$, then
\[
\frac{1}{q_k^*(\sigma) (q_{k+1}^*(\sigma) + q_k^*(\sigma))} \leq (-1)^k ( \varphi (\sigma) - \varphi_k (\sigma) ) < \frac{1}{q_k^* (\sigma) q_{k+1}^*(\sigma)}
\]
for each $0 \leq k \leq n-2$, where the first inequality is strict unless $k = n-2$ and $\sigma_n = 1$. If $\sigma \in \Sigma_n$ for some $n \in \mathbb{N}$, then
\[
\varphi (\sigma) - \varphi_{n-1} (\sigma) = \frac{(-1)^{n-1}}{q_{n-1}^* (\sigma) q_n^*(\sigma)}
\quad \text{and} \quad
\varphi(\sigma) - \varphi_k(\sigma) = 0 \quad \text{for all } k \geq n.
\]

For every $\sigma \in \Sigma$,}
\[
0 \leq (-1)^k (\varphi (\sigma) - \varphi_k (\sigma)) q_k^* (\sigma) \leq \frac{1}{F_{k+2}} \leq \frac{1}{\gamma^k}
\]
for all $k \in \mathbb{N}_0$.

\item \label{reversal identity}
Let $n \in \mathbb{N}$ and $\sigma \coloneqq (\sigma_1, \sigma_2, \dotsc) \in \Sigma$. If $\sigma \in \Sigma_k$ for some $k \geq n$ or $\sigma \in \Sigma_\infty$, then
\[
\frac{q_{n-1}^*(\sigma)}{q_n^*(\sigma)} = [\sigma_n, \dotsc, \sigma_1].
\]
\end{enumerate}
\end{proposition}

{
We remark that the equality in the first inequality of part~(\ref{metric properties(v)}) can occur only for non-realizable sequences. Indeed, if $\sigma \in \Sigma_n$ with $n \geq 2$ and $\sigma_n = 1$, then $\varphi(\sigma) = \varphi(\widetilde{\sigma})$ for $\widetilde{\sigma} \coloneqq (\sigma_1, \dotsc, \sigma_{n-2}, \sigma_{n-1}+1) \in \Sigma_{n-1}$, and by part~(\ref{metric properties(i)}), we have $q_{n-1}^*(\widetilde{\sigma}) = q_{n-1}^*(\sigma) + q_{n-2}^*(\sigma)$. Applying the identity for $\varphi(\widetilde{\sigma}) - \varphi_{n-2}(\widetilde{\sigma})$ in part~(\ref{metric properties(v)}), we obtain
\[
(-1)^{n-2}(\varphi(\sigma) - \varphi_{n-2}(\sigma)) = \frac{1}{q_{n-2}^*(\sigma)(q_{n-1}^*(\sigma) + q_{n-2}^*(\sigma))}.
\]
}

It is well known that every non-integer rational number has two distinct finite continued fraction expressions. Specifically, if $x \in (0,1) \cap \mathbb{Q}$, then
\[
[d_1(x), \dotsc, d_{n-1}(x), d_n(x)]
=
[d_1(x), \dotsc, d_{n-1}(x), d_n(x)-1, 1]
\]
(see, e.g., \cite[Theorem~162]{HW08}). The expression on the left is the CF expansion generated by \eqref{CF algorithm 1} and~\eqref{CF algorithm 2}, whereas the digit sequence corresponding to the expression on the right is not realizable.

The following proposition reformulates the preceding discussion in terms of the maps $f$ and $\varphi$. It describes the fibers of $\varphi$ in the rational and irrational cases.

\begin{proposition} [{\cite[Proposition~2.4]{Ahn25}}] \label{inverse image of phi}
Let $x \in [0,1)$. Then the following hold.
\begin{enumerate}[label = \upshape(\roman*), ref = \roman*, leftmargin=*, widest=ii]
\item \label{inverse image of phi(i)}
If $x \in (0,1) \cap \mathbb{Q}$, then $\varphi^{-1}(\{ x \}) = \{ \sigma, \sigma' \}$, where for some $n \in \mathbb{N}$,
\begin{align*}
\sigma &\coloneqq (d_1(x), d_2(x), \dotsc, d_{n-1}(x), d_n(x)) = f(x) \in \Sigma_n \cap \Sigma_{\realizable}, \\
\sigma' &\coloneqq (d_1(x), d_2(x), \dotsc, d_{n-1}(x), d_n(x)-1, 1) \in \Sigma_{n+1} \setminus \Sigma_{\realizable}.
\end{align*}
\item \label{inverse image of phi(ii)}
If $x \in \{ 0 \} \cup \mathbb{I}$, then $\varphi^{-1}(\{ x \}) = \{ \sigma \}$, where $\sigma \coloneqq f(x)$. Specifically, $\sigma \in \Sigma_\infty$ if $x \in \mathbb{I}$, and $\sigma = (\infty, \infty, \dotsc) \in \Sigma_0$ if $x=0$.
\end{enumerate}
{In particular, $\varphi(f(x)) = x$ for every $x \in [0,1)$.}
\end{proposition}

The next proposition provides a complete characterization of realizable CF sequences in the symbolic space $\Sigma$. The result follows directly from Proposition \ref{inverse image of phi}, so the proof is omitted.

\begin{proposition} \label{CF realizable sequence}
Let $\sigma \in \Sigma$. Then $\sigma \in \Sigma_{\realizable}$ if and only if one of the following holds.
\begin{enumerate}[label = \upshape(\roman*), ref = \roman*, leftmargin=*, widest=iii]
\item
$\sigma \in \Sigma_0$.
\item
$\sigma \in \Sigma_n$ for some $n \in \mathbb{N}$ with $\sigma_n > 1$.
\item
$\sigma \in \Sigma_\infty$.
\end{enumerate}
\end{proposition}

{In particular, if $x \in [0,1)$ has a finite CF expansion of length $n \in \mathbb{N}$, then $d_n(x) > 1$, as mentioned in Section~\ref{introduction}.}

Let $\sigma \in \Sigma_n$ for some $n \in \mathbb{N}$. The {\em cylinder set} determined by $\sigma$ is defined by
\[
\Upsilon_\sigma \coloneqq \{ \upsilon \in \Sigma \colon \upsilon^{(n)} = \sigma \}.
\]
The corresponding {\em fundamental interval} is defined by
\[
I_\sigma \coloneqq \{ x \in [0,1) : d_k(x) = \sigma_k \text{ for all } 1 \leq k \leq n \} = f^{-1} (\Upsilon_\sigma).
\]
For convenience, we also define the sequence $\widehat{\sigma} \in \Sigma_n$ by
\[
\widehat{\sigma} = (\sigma_1, \dotsc, \sigma_{n-1}, \sigma_n+1).
\]
The following classical result gives an explicit description of the fundamental interval $I_\sigma$.

\begin{proposition}[{\cite[Theorem~1.2.2]{IK02}}] \label{I sigma}
Let $n \in \mathbb{N}$ and $\sigma \coloneqq (\sigma_k)_{k \in \mathbb{N}} \in \Sigma_n$. Then
\begin{align*}
I_\sigma =
\begin{cases}
(\varphi (\widehat{\sigma}), \varphi (\sigma)], &\text{if $n$ is odd}; \\
[\varphi (\sigma), \varphi (\widehat{\sigma})), &\text{if $n$ is even},
\end{cases}
\quad \text{or} \quad
I_\sigma =
\begin{cases}
(\varphi (\widehat{\sigma}), \varphi (\sigma)), &\text{if $n$ is odd}; \\
(\varphi (\sigma), \varphi (\widehat{\sigma})), &\text{if $n$ is even},
\end{cases}
\end{align*}
where the former case holds if $\sigma \in \Sigma_{\realizable}$ and the latter if $\sigma \notin \Sigma_{\realizable}$. Consequently, in either case,
\begin{align} \label{length of I sigma}
\mathcal{L} (I_\sigma) = |\varphi (\sigma) - \varphi (\widehat{\sigma})| = \frac{1}{q_n^* (\sigma) (q_n^* (\sigma) + q_{n-1}^* (\sigma) )}.
\end{align}
\end{proposition}

The following example illustrates why the endpoint $\varphi (\sigma)$ is excluded when $\sigma \notin \Sigma_{\realizable}$.

\begin{example} \label{endpoint example}
Consider the sequences $\sigma \coloneqq (2) \in \Sigma_1 \cap \Sigma_{\realizable}$ and $\sigma' \coloneqq (1,1) \in \Sigma_2 \setminus \Sigma_{\realizable}$. Both are mapped to the same value under $\varphi$, with $\varphi (\sigma) = \varphi (\sigma') = \frac{1}{2}$. By definition, $I_\sigma=\left(\frac{1}{3},\frac{1}{2} \right]$ and $I_{\sigma'}=\left(\frac{1}{2},\frac{2}{3}\right)$, so $I_\sigma$ contains the endpoint $\varphi (\sigma)$, whereas $I_{\sigma'}$ does not contain $\varphi (\sigma')$.
\end{example}

The following proposition summarizes the main results of \cite{Ahn25} concerning the topology of the symbolic space $\Sigma$.

\begin{proposition} [{\cite{Ahn25}}] \label{main results of Ahn25}
There exists a topology $\mathcal{T}_\Sigma$ on $\Sigma$ such that the topological space $(\Sigma, \mathcal{T}_\Sigma)$ is compact metrizable. With this topology, the following hold.
\begin{enumerate}[label = \upshape(\roman*), ref = \roman*, leftmargin=*, widest=ii]
\item \label{cylinder set is compact}
For each $n \in \mathbb{N}$ and $\sigma \in \Sigma_n$, both the cylinder set $\Upsilon_\sigma$ and its complement $\Sigma \setminus \Upsilon_\sigma$ are closed and compact in $(\Sigma, \mathcal{T}_\Sigma)$.

\item \label{g is quotient}
Equip $\mathbb{N}_\infty^{\mathbb{N}}$ with the product topology, where $\mathbb{N}_\infty$ denotes the one-point compactification of the discrete space $\mathbb{N}$. Define $g \colon \mathbb{N}_\infty^{\mathbb{N}} \to (\Sigma, \mathcal{T}_\Sigma)$ as follows. For $\sigma \coloneqq (\sigma_k)_{k\in\mathbb N} \in \mathbb{N}_\infty^{\mathbb N}$, if there exists a smallest $n \in \mathbb{N}_0$ such that $\sigma_{n+1}=\infty$, set
\[
g(\sigma) \coloneqq (\sigma_1,\dotsc,\sigma_n,\infty,\infty,\dotsc) \in \Sigma_n.
\]
If no such $n$ exists (that is, if $\sigma_k \in \mathbb{N}$ for all $k$), set $g(\sigma) \coloneqq \sigma \in \Sigma_\infty$. Then $g$ is a quotient map.

\item \label{phi is continuous}
The map $\varphi \colon (\Sigma, \mathcal{T}_\Sigma) \to [0,1]$ is continuous.
\item \label{f continuity theorem}
The continued fraction digit coding map $f \colon [0,1) \to (\Sigma, \mathcal{T}_\Sigma)$ satisfies the following.
	\begin{enumerate}[label = \upshape(\alph*), ref = \alph*, leftmargin=*, widest=b]
	\item \label{f is continuous at irrational}
	$f$ is continuous at every irrational point and at the left endpoint $0$.
	\item \label{f is discontinuous at rational}
	At each rational point $x \in (0,1) \cap \mathbb{Q}$, the map $f$ is discontinuous. More precisely, if $\varphi^{-1} ( \{ x \} ) = \{ \sigma, \tau \}$, then
\[
\lim_{\substack{t \to x \\ t \in I_\sigma}} f(t) = \sigma
\quad \text{and} \quad
\lim_{\substack{t \to x \\ t \notin I_\sigma}} f(t) = \tau.
\]
	\end{enumerate}
\end{enumerate}
\end{proposition}

The following lemma is an immediate consequence of the construction in \cite{Ahn25}. Although it was not stated explicitly there, we record it for later reference.

\begin{lemma} \label{varphi n is continuous}
For each $n \in \mathbb{N}_0$, the map $\varphi_n \colon (\Sigma,\mathcal{T}_\Sigma) \to [0,1]$ is continuous.
\end{lemma}

\begin{proof}
Since $\varphi_0$ is constantly $0$ by definition, it is continuous.

Now, let $n \in \mathbb{N}$ and define
\[
\widetilde{\varphi}_n(\sigma) \coloneqq [\sigma_1,\dotsc,\sigma_n,\infty,\infty,\dotsc],
\quad
\sigma \in \mathbb{N}_\infty^{\mathbb{N}}.
\]
Since finite continued fractions depend continuously on their digits (with the convention $1/\infty=0$), the map $\widetilde{\varphi}_n$ is continuous. By definition, $\widetilde{\varphi}_n=\varphi_n\circ g$. Since $g$ is a quotient map by Proposition \ref{main results of Ahn25}(\ref{g is quotient}), the continuity of $\widetilde{\varphi}_n$ implies the continuity of $\varphi_n$.
\end{proof}

\section{Exact oscillation of the error-sum functions on fundamental intervals} \label{Auxiliary results}

{
In this section, we determine the oscillation of $\mathcal{E}$ and $P$ on each fundamental interval (Lemma~\ref{sup of difference of E in I sigma}). In order to treat the two functions simultaneously, for each $\lambda \in \{0,1\}$, we define $\mathcal{E}_\lambda \colon \mathbb{R} \to \mathbb{R}$ by
\[
\mathcal{E}_\lambda(x) \coloneqq \sum_{n \geq 0} (q_n(x))^\lambda \left( x - \frac{p_n(x)}{q_n(x)} \right).
\]
Then $\mathcal{E}_0 = \mathcal{E}$ and $\mathcal{E}_1 = P$. Some of the results in this section are analogues of those established for Pierce expansions in \cite{Ahn23}. (See Section~\ref{introduction} for the main differences.)
}

\begin{lemma} \label{f image of fundamental interval}
Let $\sigma \in \Sigma_n$ for some $n \in \mathbb{N}$. Then the following hold.
\begin{enumerate}[label = \upshape(\roman*), ref = \roman*, leftmargin=*, widest=ii]
\item \label{f image of fundamental interval(i)}
We have the strict inclusion $f(I_\sigma) \subsetneq \Upsilon_\sigma$. More explicitly,
\[
\Upsilon_\sigma \setminus f(I_\sigma) = \bigcup_{m \geq n} \left\{ \tau \coloneqq (\tau_k)_{k \in \mathbb{N}} \in \Sigma_m : \tau^{(n)} = \sigma \text{ and } \tau_m = 1 \right\}.
\]
\item \label{f image of fundamental interval is dense}
We have $\overline{f(I_\sigma)} = \Upsilon_\sigma$; that is, the image $f(I_\sigma)$ is dense in the cylinder set $\Upsilon_\sigma$.
\end{enumerate}
\end{lemma}

\begin{proof}
Part (\ref{f image of fundamental interval(i)}) follows directly from Proposition \ref{CF realizable sequence}, so it remains to prove part (\ref{f image of fundamental interval is dense}). Let $\sigma \coloneqq (\sigma_k)_{k \in \mathbb{N}}$. By definition, $f(I_\sigma) \subseteq \Upsilon_\sigma$. Moreover, Proposition \ref{main results of Ahn25}(\ref{cylinder set is compact}) shows that $\Upsilon_\sigma$ is closed in $(\Sigma, \mathcal{T}_\Sigma)$. Therefore, it suffices to show that every point in $\Upsilon_\sigma \setminus f(I_\sigma)$ is a limit point of $f(I_\sigma)$. Take any $\upsilon \coloneqq (\upsilon_k)_{k \in \mathbb{N}} \in \Upsilon_\sigma \setminus f(I_\sigma)$. Then, by part (\ref{f image of fundamental interval(i)}), there exists $m \geq n$ such that
\[
\upsilon = (\upsilon_1, \dotsc, \upsilon_m) \in \Sigma_m, \quad \upsilon^{(n)} = \sigma, \quad \upsilon_m = 1.
\]
Define a sequence $(\bm{\tau}_k)_{k \in \mathbb{N}}$ in $\Sigma$ by
\[
\bm{\tau}_k \coloneqq (\upsilon_1, \dotsc, \upsilon_m, k+1) \in \Sigma_{m+1}, \quad k \in \mathbb{N}.
\]
By Proposition \ref{CF realizable sequence}, each $\bm{\tau}_k$ belongs to $f(I_\sigma)$, since its last digit $k+1 > 1$ ensures realizability. We have $\bm{\tau}_k \to \upsilon$ as $k \to \infty$ in the product space $\mathbb{N}_\infty^{\mathbb{N}}$. Using the continuous map $g \colon \mathbb{N}_\infty^{\mathbb{N}} \to (\Sigma, \mathcal{T}_{\Sigma})$ from Proposition \ref{main results of Ahn25}(\ref{g is quotient}), we obtain $g(\bm{\tau}_k) \to g(\upsilon)$ as $k \to \infty$ in $(\Sigma, \mathcal{T}_\Sigma)$. Moreover, $g(\bm{\tau}_k) = \bm{\tau}_k$ and $g(\upsilon) = \upsilon$. Thus, $\upsilon \in \overline{f(I_\sigma)}$, as required.
\end{proof}

{
The following lemma describes the image of a cylinder set under $\varphi$. Notice that the image is the closure $\overline{I_\sigma}$ of the fundamental interval, not $I_\sigma$ itself.

\begin{lemma} \label{phi of cylinder}
Let $\sigma \in \Sigma_n$ for some $n \in \mathbb{N}$. Then $\varphi(\Upsilon_\sigma) = \overline{I_\sigma}$. Consequently, for every $\tau \in \Upsilon_\sigma$, we have
\[
|\varphi(\tau) - \varphi(\sigma)| \leq \mathcal{L}(I_\sigma),
\]
where the equality holds if $\tau = \sigma^+ \coloneqq (\sigma_1, \dotsc, \sigma_n, 1) \in \Upsilon_\sigma \cap \Sigma_{n+1}$.
\end{lemma}

\begin{proof}
By Proposition~\ref{inverse image of phi}, we have $\varphi(f(x)) = x$ for each $x \in [0,1)$, and so $\varphi(f(I_\sigma)) = I_\sigma$. By Lemma~\ref{f image of fundamental interval}(\ref{f image of fundamental interval is dense}), $\Upsilon_\sigma$ is the closure of $f(I_\sigma)$ in $(\Sigma, \mathcal{T}_\Sigma)$. Since $\varphi$ is continuous by Proposition~\ref{main results of Ahn25}(\ref{phi is continuous}), it follows that
\[
\varphi(\Upsilon_\sigma) = \varphi\big(\overline{f(I_\sigma)}\big) \subseteq \overline{\varphi(f(I_\sigma))} = \overline{I_\sigma}.
\]
On the other hand, since $\Upsilon_\sigma$ is compact by Proposition~\ref{main results of Ahn25}(\ref{cylinder set is compact}), the set $\varphi(\Upsilon_\sigma)$ is a compact subset of $\mathbb{R}$ containing $I_\sigma$. Thus $\overline{I_\sigma} \subseteq \varphi(\Upsilon_\sigma)$, and we conclude that $\varphi(\Upsilon_\sigma) = \overline{I_\sigma}$.

Now, by Proposition~\ref{I sigma}, $\overline{I_\sigma}$ is a closed interval of length $\mathcal{L}(I_\sigma)$ with endpoints $\varphi(\sigma)$ and $\varphi(\widehat{\sigma})$. Hence $|\varphi(\tau) - \varphi(\sigma)| \leq \mathcal{L}(I_\sigma)$ for every $\tau \in \Upsilon_\sigma$. Finally, since $\varphi(\sigma^+) = [\sigma_1, \dotsc, \sigma_n, 1] = [\sigma_1, \dotsc, \sigma_{n-1}, \sigma_n + 1] = \varphi(\widehat{\sigma})$, the equality holds for $\tau = \sigma^+$. This completes the proof.
\end{proof}

The following lemma shows the monotonicity of the approximation errors and of the scaled approximation errors for all CF sequences, including non-realizable ones. When $\lambda = 0$ and $\sigma \in \Sigma_\infty$, the lemma reduces to the classical fact that $|x - p_k(x)/q_k(x)|$ decreases in $k$ (see \cite[Theorem~164]{HW08}).

\begin{lemma} \label{scaled errors}
Let $\sigma \in \Sigma$ and $\lambda \in \{0,1\}$, and define
\[
r_k^{(\lambda)}(\sigma) \coloneqq (-1)^k (\varphi(\sigma) - \varphi_k(\sigma)) (q_k^*(\sigma))^\lambda, \quad k \in \mathbb{N}_0.
\]
Then $(r_k^{(\lambda)}(\sigma))_{k \in \mathbb{N}_0}$ is a non-increasing sequence of non-negative real numbers converging to $0$.
\end{lemma}

\begin{proof}
Put $q_k \coloneqq q_k^*(\sigma)$, $r_k \coloneqq r_k^{(1)}(\sigma)$, and $a_k \coloneqq r_k^{(0)}(\sigma) = r_k/q_k$ for each $k \in \mathbb{N}_0$. By Proposition~\ref{metric properties}(\ref{metric properties(v)}), we have $0 \leq r_k \leq 1/F_{k+2}$, and so $r_k \to 0$ as $k \to \infty$. Since $q_k \geq 1$, the same holds for $a_k$.

We first show that $r_{k+1} \leq r_k$ for each $k \in \mathbb{N}_0$. If $\sigma \in \Sigma_m$ for some $m \leq k+1$, then $\varphi_{k+1}(\sigma) = \varphi(\sigma)$ by Proposition~\ref{metric properties}(\ref{metric properties(v)}), so that $r_{k+1} = 0 \leq r_k$. Otherwise, either $\sigma \in \Sigma_\infty$ or $\sigma \in \Sigma_m$ for some $m \geq k+2$; in particular, $\sigma_{k+2} \in \mathbb{N}$. In either case, by Proposition~\ref{metric properties}(\ref{metric properties(v)}), applied with the index $k+1$ for the upper bound and with the index $k$ for the lower bound, we have
\[
r_{k+1} \leq q_{k+1} \cdot \frac{1}{q_{k+1} q_{k+2}} = \frac{1}{q_{k+2}}
\quad \text{and} \quad
r_k \geq q_k \cdot \frac{1}{q_k (q_{k+1} + q_k)} = \frac{1}{q_{k+1} + q_k}.
\]
Since $\sigma_{k+2} \geq 1$, Proposition~\ref{metric properties}(\ref{metric properties(i)}) tells us that $q_{k+2} = \sigma_{k+2} q_{k+1} + q_k \geq q_{k+1} + q_k$. Therefore, $r_{k+1} \leq r_k$.

Next, by Proposition~\ref{metric properties}(\ref{metric properties(i)}), we have $q_{k+1} \geq q_k \geq 1$ for each $k \in \mathbb{N}_0$. It follows that
\[
a_{k+1} = \frac{r_{k+1}}{q_{k+1}} \leq \frac{r_k}{q_k} = a_k
\]
for each $k \in \mathbb{N}_0$. This proves the lemma.
\end{proof}

For $\lambda \in \{0,1\}$, we define the {\em error-sum functions of CF sequences} $\mathcal{E}_\lambda^* \colon \Sigma \to \mathbb{R}$ by
\[
\mathcal{E}_\lambda^*(\sigma) \coloneqq \sum_{n \geq 0} (q_n^*(\sigma))^\lambda (\varphi (\sigma) - \varphi_n(\sigma) )
\]
for each $\sigma \in \Sigma$. We also write $\mathcal{E}^* \coloneqq \mathcal{E}_0^*$ and $P^* \coloneqq \mathcal{E}_1^*$, and call them the {\em unweighted} and the {\em relative error-sum functions of CF sequences}, respectively. By Proposition~\ref{metric properties}(\ref{metric properties(v)}), we have
\begin{align} \label{uniform bound for terms}
\left| (q_n^*(\sigma))^\lambda (\varphi (\sigma) - \varphi_n(\sigma)) \right| \leq \frac{1}{F_{n+2}} \leq \frac{1}{\gamma^n}
\end{align}
for all $\sigma \in \Sigma$, $n \in \mathbb{N}_0$, and $\lambda \in \{0,1\}$. Hence the series defining $\mathcal{E}_\lambda^*$ converges absolutely and uniformly on $\Sigma$, and $\mathcal{E}_\lambda^*$ is well defined.

\begin{lemma} \label{commutative diagram for E}
For each $\lambda \in \{0,1\}$, we have $\mathcal{E}_\lambda = \mathcal{E}_\lambda^* \circ f$ on $[0,1)$; that is, the following diagram commutes:
\begin{center}
\begin{tikzcd}[column sep=small]
[0,1) \arrow{rr}{\mathcal{E}_\lambda} \arrow[swap, shift right=.75ex]{dr}{f}& &\mathbb{R} \\
& \Sigma \arrow[swap]{ur}{\mathcal{E}_\lambda^*}
\end{tikzcd}
\end{center}
In particular, $\mathcal{E} = \mathcal{E}^* \circ f$ and $P = P^* \circ f$ on $[0,1)$.
\end{lemma}

\begin{proof}
Let $x \in [0,1)$, and let $\sigma \coloneqq f(x) = (d_1(x),d_2(x),\dotsc) \in \Sigma$. By Proposition~\ref{inverse image of phi}, we have $\varphi (\sigma) = x$. By the definitions of $\varphi_n$ and of the convergents, we have $\varphi_n(\sigma) = [d_1(x), \dotsc, d_n(x), \infty, \infty, \dotsc] = p_n(x)/q_n(x)$ for each $n \in \mathbb{N}_0$. Since $\gcd(p_n(x),q_n(x))=1$, the reduced denominator of $\varphi_n(\sigma)$ is $q_n(x)$. Hence $q_n^*(\sigma)=q_n(x)$. Therefore,
\[
(\mathcal{E}_\lambda^* \circ f)(x) = \sum_{n \geq 0} (q_n^*(\sigma))^\lambda (\varphi (\sigma) - \varphi_n(\sigma)) = \sum_{n \geq 0} (q_n(x))^\lambda \left( x - \frac{p_n(x)}{q_n(x)} \right) = \mathcal{E}_\lambda(x),
\]
as desired.
\end{proof}
}

However, the converse identity $\mathcal{E} \circ \varphi = \mathcal{E}^*$ does not hold in general.

\begin{example}
Let $\sigma \coloneqq (1,1) \in \Sigma_2 \setminus \Sigma_{\realizable}$. We have $\varphi (\sigma) = \frac{1}{1+\frac{1}{1}} = \frac{1}{2} = [2]$. Hence, $(\mathcal{E} \circ \varphi)(\sigma) = \mathcal{E}([2]) = \left( \frac{1}{2} - 0 \right) + \left( \frac{1}{2} - \frac{1}{2} \right) = \frac{1}{2}$. On the other hand, $\mathcal{E}^*(\sigma) = \left( \frac{1}{2} - 0 \right) + \left( \frac{1}{2} - 1 \right) = 0$. Thus, the identity $\mathcal{E} \circ \varphi=\mathcal{E}^*$ fails in general. The example shows that this can happen when $\sigma \notin \Sigma_{\realizable}$. {Similarly, $(P \circ \varphi)(\sigma) = P([2]) = \frac{1}{2} + \left( 2 \cdot \frac{1}{2} - 1 \right) = \frac{1}{2}$, whereas $P^*(\sigma) = \frac{1}{2} + \left( \frac{1}{2} - 1 \right) \cdot 1 = 0$, since $q_1^*(\sigma) = 1$.}
\end{example}

{
We now establish the continuity of $\mathcal{E}_\lambda^*$ on $(\Sigma, \mathcal{T}_\Sigma)$. For $\lambda = 0$, this is a CF analogue of \cite[Lemma~3.15]{Ahn23}. For $\lambda = 1$, however, some care is needed, since the weights $q_k^*$ are not continuous on $\Sigma$. For instance, we have $(m) \to (\infty, \infty, \dotsc)$ as $m \to \infty$ in $(\Sigma, \mathcal{T}_\Sigma)$ by Proposition~\ref{main results of Ahn25}(\ref{g is quotient}), whereas $q_1^*((m)) = m$ for each $m \in \mathbb{N}$ and $q_1^*((\infty, \infty, \dotsc)) = 1$.

\begin{lemma} \label{continuity of symbolic error sums}
Let $\lambda \in \{0,1\}$. For each $k \in \mathbb{N}_0$, the map
\[
u_k^{(\lambda)} \coloneqq (q_k^*)^\lambda (\varphi - \varphi_k) \colon (\Sigma, \mathcal{T}_\Sigma) \to \mathbb{R}
\]
is continuous. Consequently, $\mathcal{E}_\lambda^* \colon (\Sigma, \mathcal{T}_\Sigma) \to \mathbb{R}$ is continuous. In particular, both $\mathcal{E}^*$ and $P^*$ are continuous.
\end{lemma}

\begin{proof}
If $\lambda = 0$, then $u_k^{(0)} = \varphi - \varphi_k$ is continuous by Proposition~\ref{main results of Ahn25}(\ref{phi is continuous}) and Lemma~\ref{varphi n is continuous}. Assume $\lambda = 1$. Since $u_0^{(1)} = \varphi$ is continuous, we may assume that $k \in \mathbb{N}$. Since $g$ is a quotient map by Proposition~\ref{main results of Ahn25}(\ref{g is quotient}), it suffices to show that $u_k^{(1)} \circ g$ is continuous on $\mathbb{N}_\infty^{\mathbb{N}}$. Moreover, since $\mathbb{N}_\infty^{\mathbb{N}}$ is metrizable, it is enough to establish the sequential continuity. Let $\rho_j \coloneqq (\rho_{j,i})_{i \in \mathbb{N}}$, $j \in \mathbb{N}$, be a sequence in $\mathbb{N}_\infty^{\mathbb{N}}$ converging to $\rho \coloneqq (\rho_i)_{i \in \mathbb{N}}$ in the product topology, i.e., $\rho_{j,i} \to \rho_i$ in $\mathbb{N}_\infty$ as $j \to \infty$ for each $i \in \mathbb{N}$. Put $\sigma \coloneqq g(\rho)$ and $\sigma_{(j)} \coloneqq g(\rho_j)$ for each $j \in \mathbb{N}$. Note that $\varphi \circ g$ and $\varphi_k \circ g$ are continuous by Proposition~\ref{main results of Ahn25}(\ref{g is quotient}),(\ref{phi is continuous}) and Lemma~\ref{varphi n is continuous}. We consider two cases.

Case I. Assume $\rho_1, \dotsc, \rho_k \in \mathbb{N}$. Since each point of $\mathbb{N}$ is isolated in $\mathbb{N}_\infty$, there exists $J \in \mathbb{N}$ such that $\rho_{j,i} = \rho_i$ for all $1 \leq i \leq k$ and all $j \geq J$. For such $j$, the first $k$ terms of $\sigma_{(j)}$ and $\sigma$ coincide, and hence $q_k^*(\sigma_{(j)}) = q_k^*(\sigma)$. Therefore,
\[
u_k^{(1)}(\sigma_{(j)}) = q_k^*(\sigma) \big( (\varphi \circ g)(\rho_j) - (\varphi_k \circ g)(\rho_j) \big) \to q_k^*(\sigma) \big( \varphi(\sigma) - \varphi_k(\sigma) \big) = u_k^{(1)}(\sigma)
\]
as $j \to \infty$.

Case II. Assume $\rho_i = \infty$ for some $1 \leq i \leq k$. Let $m+1$ be the smallest such $i$, so that $0 \leq m \leq k-1$ and $\sigma = g(\rho) \in \Sigma_m$. Then $\varphi_k(\sigma) = \varphi(\sigma)$, and so $u_k^{(1)}(\sigma) = 0$. Fix $j \in \mathbb{N}$. If $\sigma_{(j)} \in \Sigma_{m'}$ for some $m' \leq k$, then $u_k^{(1)}(\sigma_{(j)}) = 0$. Otherwise, either $\sigma_{(j)} \in \Sigma_\infty$ or $\sigma_{(j)} \in \Sigma_{m'}$ for some $m' \geq k+1$; in particular, the $(m+1)$th term of $\sigma_{(j)}$ is $\rho_{j,m+1} \in \mathbb{N}$. In this case, by Proposition~\ref{metric properties}(\ref{metric properties(v)}) and (\ref{metric properties(i)}), we have
\[
|u_k^{(1)}(\sigma_{(j)})| \leq q_k^*(\sigma_{(j)}) \cdot \frac{1}{q_k^*(\sigma_{(j)}) q_{k+1}^*(\sigma_{(j)})} = \frac{1}{q_{k+1}^*(\sigma_{(j)})} \leq \frac{1}{q_{m+1}^*(\sigma_{(j)})} \leq \frac{1}{\rho_{j,m+1}},
\]
where the last inequality follows from $q_{m+1}^*(\sigma_{(j)}) = \rho_{j,m+1} q_m^*(\sigma_{(j)}) + q_{m-1}^*(\sigma_{(j)}) \geq \rho_{j,m+1}$. Thus, with the convention $1/\infty = 0$, we have $|u_k^{(1)}(\sigma_{(j)})| \leq 1/\rho_{j,m+1}$ for all $j \in \mathbb{N}$. Since $\rho_{j,m+1} \to \rho_{m+1} = \infty$ in $\mathbb{N}_\infty$, we conclude that $u_k^{(1)}(\sigma_{(j)}) \to 0 = u_k^{(1)}(\sigma)$ as $j \to \infty$.

In either case, we have $u_k^{(1)}(g(\rho_j)) \to u_k^{(1)}(g(\rho))$ as $j \to \infty$, which proves the continuity of $u_k^{(1)}$.

Finally, we have $\mathcal{E}_\lambda^* = \sum_{k \geq 0} u_k^{(\lambda)}$, where each $u_k^{(\lambda)}$ is continuous and $\sup_{\Sigma} |u_k^{(\lambda)}| \leq \gamma^{-k}$ by \eqref{uniform bound for terms}. By the Weierstrass M-test, the series converges uniformly on $\Sigma$, and hence $\mathcal{E}_\lambda^*$ is continuous. Thus the lemma follows.
\end{proof}

\begin{lemma} \label{E star value difference of sigma and sigma prime}
Let $\sigma \coloneqq (\sigma_k)_{k=1}^n \in \Sigma_n \cap \Sigma_{\realizable}$ for some $n \in \mathbb{N}$, and define $\sigma' \coloneqq (\sigma_1, \dotsc, \sigma_{n-1}, \sigma_n-1, 1) \in \Sigma_{n+1} \setminus \Sigma_{\realizable}$. Then, for each $\lambda \in \{0,1\}$,
\[
\mathcal{E}_\lambda^*(\sigma') = \mathcal{E}_\lambda^*(\sigma) + \frac{(-1)^n}{q_n^* (\sigma)(q_n^* (\sigma) - q_{n-1}^* (\sigma))^{1-\lambda}}.
\]
In particular,
\[
\mathcal{E}^*(\sigma') = \mathcal{E}^*(\sigma) + \frac{(-1)^n}{q_n^* (\sigma)(q_n^* (\sigma) - q_{n-1}^* (\sigma))}
\quad \text{and} \quad
P^*(\sigma') = P^*(\sigma) + \frac{(-1)^n}{q_n^*(\sigma)},
\]
where the latter identity is due to Ridley and Petruska \cite[Lemma~1.1]{RP00}.
\end{lemma}

\begin{proof}
Observe that $\varphi(\sigma') = \varphi(\sigma)$, that $\varphi_k(\sigma') = \varphi_k(\sigma)$ and $q_k^*(\sigma') = q_k^*(\sigma)$ for $0 \leq k \leq n-1$ by Proposition~\ref{metric properties}(\ref{metric properties(ii)}), that $\varphi_k(\sigma) = \varphi(\sigma)$ for $k \geq n$, and that $\varphi_k(\sigma') = \varphi(\sigma')$ for $k \geq n+1$. Hence all terms of $\mathcal{E}_\lambda^*(\sigma')$ and $\mathcal{E}_\lambda^*(\sigma)$ coincide except for the terms with index $n$, and
\[
\mathcal{E}_\lambda^*(\sigma') - \mathcal{E}_\lambda^*(\sigma) = (q_n^*(\sigma'))^\lambda (\varphi_n(\sigma) - \varphi_n(\sigma')).
\]
By Proposition~\ref{metric properties}(\ref{metric properties(ii)}), we have $q_n^*(\sigma') = q_n^*(\sigma) - q_{n-1}^*(\sigma)$ and
\[
\varphi_n(\sigma') = \frac{p_n^*(\sigma')}{q_n^*(\sigma')} = \frac{p_n^* (\sigma) - p_{n-1}^* (\sigma)}{q_n^* (\sigma) - q_{n-1}^* (\sigma)}.
\]
Hence,
\begin{align*}
\varphi_{n} (\sigma) - \varphi_n (\sigma')
&= \frac{p_n^* (\sigma)}{q_n^* (\sigma)} - \frac{p_n^* (\sigma) - p_{n-1}^* (\sigma)}{q_n^* (\sigma) - q_{n-1}^* (\sigma)} \\
&= \frac{- p_n^* (\sigma) q_{n-1}^* (\sigma) + p_{n-1}^* (\sigma) q_n^* (\sigma)}{q_n^* (\sigma) (q_n^* (\sigma) - q_{n-1}^* (\sigma))}
= \frac{(-1)^n}{q_n^* (\sigma) (q_n^* (\sigma) - q_{n-1}^* (\sigma))},
\end{align*}
where the last equality follows from Proposition \ref{metric properties}(\ref{metric properties(iv)}). Multiplying both sides by $(q_n^*(\sigma'))^\lambda = (q_n^*(\sigma) - q_{n-1}^*(\sigma))^\lambda$, we obtain the desired identity.
\end{proof}

The following two lemmas are the main results of this section. For $n \in \mathbb{N}$, $\sigma \in \Sigma_n$, and $\lambda \in \{0,1\}$, put
\begin{align} \label{definition of W}
W_\lambda(\sigma) \coloneqq \sum_{k=0}^{n} (q_k^*(\sigma))^\lambda,
\quad \text{so that} \quad
W_0(\sigma) = n+1
\quad \text{and} \quad
W_1(\sigma) = \sum_{k=0}^{n} q_k^*(\sigma).
\end{align}

\begin{lemma} \label{bounds for E star in Upsilon sigma}
Let $\sigma \in \Sigma_n$ for some $n \in \mathbb{N}$, and let $\lambda \in \{0,1\}$. Then, for every $\tau \in \Upsilon_\sigma$, we have
\begin{align} \label{two-sided bound}
0 \leq (-1)^n \big( \mathcal{E}_\lambda^*(\tau) - \mathcal{E}_\lambda^*(\sigma) \big) \leq W_\lambda(\sigma) \cdot \mathcal{L}(I_\sigma),
\end{align}
where the lower bound is attained at $\tau = \sigma$ and the upper bound is attained at $\tau = \sigma^+ \coloneqq (\sigma_1, \dotsc, \sigma_n, 1) \in \Upsilon_\sigma \cap \Sigma_{n+1}$. Consequently, the following hold.
\begin{enumerate}[label = \upshape(\roman*), ref = \roman*, leftmargin=*, widest=ii]
\item \label{bounds for E star in Upsilon sigma 1}
If $n$ is odd, then
\[
\max_{\tau \in \Upsilon_\sigma} \mathcal{E}_\lambda^*(\tau) = \mathcal{E}_\lambda^*(\sigma)
\quad \text{and} \quad
\min_{\tau \in \Upsilon_\sigma} \mathcal{E}_\lambda^*(\tau) = \mathcal{E}_\lambda^* (\sigma) - W_\lambda(\sigma) \cdot \mathcal{L} (I_\sigma).
\]

\item
If $n$ is even, then
\[
\max_{\tau \in \Upsilon_\sigma} \mathcal{E}_\lambda^*(\tau) = \mathcal{E}_\lambda^*(\sigma) + W_\lambda(\sigma) \cdot \mathcal{L} (I_\sigma)
\quad \text{and} \quad
\min_{\tau \in \Upsilon_\sigma} \mathcal{E}_\lambda^*(\tau) = \mathcal{E}_\lambda^* (\sigma).
\]
\end{enumerate}
\end{lemma}

\begin{proof}
Fix $\tau \in \Upsilon_\sigma$, and put $r_k \coloneqq r_k^{(\lambda)}(\tau)$ for each $k \in \mathbb{N}_0$, as in Lemma~\ref{scaled errors}. Then the $k$th term of the series defining $\mathcal{E}_\lambda^*(\tau)$ equals $(-1)^k r_k$. For each $m \in \mathbb{N}_0$, we define the tail
\[
R_m \coloneqq \sum_{k \geq m} (-1)^k r_k.
\]
By Lemma~\ref{scaled errors}, the sequence $(r_k)_{k \in \mathbb{N}_0}$ is non-negative and non-increasing, and it converges to $0$. Hence, by the standard estimate for alternating series, we have
\begin{align} \label{alternating tail}
0 \leq (-1)^m R_m = r_m - r_{m+1} + r_{m+2} - \dotsb \leq r_m
\quad \text{for each } m \in \mathbb{N}_0.
\end{align}

Since $\tau^{(n)} = \sigma$, the sequences $\tau$ and $\sigma$ share the first $n$ terms, so that $\varphi_k(\tau) = \varphi_k(\sigma)$ and $q_k^*(\tau) = q_k^*(\sigma)$ for $0 \leq k \leq n$. Moreover, since $\sigma \in \Sigma_n$, we have $\varphi_k(\sigma) = \varphi(\sigma)$ for all $k \geq n$, i.e., the terms of $\mathcal{E}_\lambda^*(\sigma)$ with index $k \geq n$ vanish. It follows that
\begin{align} \label{difference formula 1}
\mathcal{E}_\lambda^*(\tau) - \mathcal{E}_\lambda^*(\sigma)
= \sum_{k=0}^{n-1} (q_k^*(\sigma))^\lambda \big( \varphi(\tau) - \varphi(\sigma) \big) + R_n.
\end{align}
Notice that the first term of $R_n$ equals $(q_n^*(\sigma))^\lambda (\varphi(\tau) - \varphi_n(\sigma)) = (q_n^*(\sigma))^\lambda (\varphi(\tau) - \varphi(\sigma))$. Hence, \eqref{difference formula 1} can be rewritten as
\begin{align} \label{difference formula 2}
\mathcal{E}_\lambda^*(\tau) - \mathcal{E}_\lambda^*(\sigma)
= W_\lambda(\sigma) \big( \varphi(\tau) - \varphi(\sigma) \big) + R_{n+1}.
\end{align}
By Proposition~\ref{metric properties}(\ref{metric properties(v)}), applied to $\tau$ with the index $n$, we also have
\begin{align} \label{sign at level n}
(-1)^n \big( \varphi(\tau) - \varphi(\sigma) \big) = (-1)^n \big( \varphi(\tau) - \varphi_n(\tau) \big) \geq 0.
\end{align}

We first prove the lower bound. Multiplying both sides of \eqref{difference formula 1} by $(-1)^n$, and using \eqref{sign at level n} and the inequality $(-1)^n R_n \geq 0$ from \eqref{alternating tail}, we obtain $(-1)^n (\mathcal{E}_\lambda^*(\tau) - \mathcal{E}_\lambda^*(\sigma)) \geq 0$.

We next prove the upper bound. Multiplying both sides of \eqref{difference formula 2} by $(-1)^n$, and using the inequality $(-1)^n R_{n+1} = -(-1)^{n+1} R_{n+1} \leq 0$ from \eqref{alternating tail}, we obtain
\[
(-1)^n \big( \mathcal{E}_\lambda^*(\tau) - \mathcal{E}_\lambda^*(\sigma) \big)
\leq W_\lambda(\sigma) \, |\varphi(\tau) - \varphi(\sigma)|
\leq W_\lambda(\sigma) \cdot \mathcal{L}(I_\sigma),
\]
where the last inequality follows from Lemma~\ref{phi of cylinder}. This proves \eqref{two-sided bound}.

It remains to consider the equality cases. If $\tau = \sigma$, the middle term of \eqref{two-sided bound} is clearly $0$. If $\tau = \sigma^+$, then $\varphi_k(\sigma^+) = \varphi(\sigma^+)$ for all $k \geq n+1$, so that $R_{n+1} = 0$ in \eqref{difference formula 2}. Moreover, we have $|\varphi(\sigma^+) - \varphi(\sigma)| = \mathcal{L}(I_\sigma)$ by Lemma~\ref{phi of cylinder}, and so $(-1)^n (\varphi(\sigma^+) - \varphi(\sigma)) = \mathcal{L}(I_\sigma)$ by \eqref{sign at level n}. Thus, $(-1)^n (\mathcal{E}_\lambda^*(\sigma^+) - \mathcal{E}_\lambda^*(\sigma)) = W_\lambda(\sigma) \cdot \mathcal{L}(I_\sigma)$.

Now, parts (i) and (ii) follow immediately from \eqref{two-sided bound} and the equality cases. This finishes the proof of the lemma.
\end{proof}

It is worth noting where the weight $(q_k^*)^\lambda$ plays a role in the above argument. In the unweighted case $\lambda = 0$, the factor $W_0(\sigma) = n+1$ in Lemma~\ref{bounds for E star in Upsilon sigma} simply counts the terms with index at most $n$, whereas, in the relative case $\lambda = 1$, this factor is replaced by $W_1(\sigma) = \sum_{k=0}^n q_k^*(\sigma)$. Apart from this, the weight matters only in the monotonicity of the scaled errors (Lemma~\ref{scaled errors}) and in the continuity at finite sequences (Lemma~\ref{continuity of symbolic error sums}).

\begin{lemma} \label{sup of difference of E in I sigma}
Let $\sigma \in \Sigma_n$ for some $n \in \mathbb{N}$, and let $\lambda \in \{0,1\}$. Then
\[
\sup_{t,u \in I_\sigma} |\mathcal{E}_\lambda(t) - \mathcal{E}_\lambda(u)| = W_\lambda(\sigma) \cdot \mathcal{L} (I_\sigma).
\]
In particular,
\[
\sup_{t,u \in I_\sigma} |\mathcal{E}(t) - \mathcal{E}(u)| = (n+1) \cdot \mathcal{L} (I_\sigma)
\quad \text{and} \quad
\sup_{t,u \in I_\sigma} |P(t) - P(u)| = \mathcal{L} (I_\sigma) \sum_{k=0}^{n} q_k^*(\sigma).
\]
\end{lemma}

\begin{proof}
By Lemma \ref{commutative diagram for E}, we have $\mathcal{E}_\lambda(t) = \mathcal{E}_\lambda^*(f(t))$ for all $t \in I_\sigma$. Since $\mathcal{E}_\lambda^*$ is continuous by Lemma~\ref{continuity of symbolic error sums} and $f(I_\sigma)$ is dense in $\Upsilon_\sigma$ by Lemma \ref{f image of fundamental interval}(\ref{f image of fundamental interval is dense}), the supremum and infimum of $\mathcal{E}_\lambda^*$ over $f(I_\sigma)$ coincide with those over $\Upsilon_\sigma$. Therefore,
\[
\sup_{t \in I_\sigma} \mathcal{E}_\lambda(t) = \max_{\tau \in \Upsilon_\sigma} \mathcal{E}_\lambda^*(\tau)
\quad \text{and} \quad
\inf_{t \in I_\sigma} \mathcal{E}_\lambda(t) = \min_{\tau \in \Upsilon_\sigma} \mathcal{E}_\lambda^*(\tau),
\]
where the maximum and the minimum exist by Lemma~\ref{bounds for E star in Upsilon sigma}. Thus, again by Lemma \ref{bounds for E star in Upsilon sigma}, we have
\[
\sup_{t,u \in I_\sigma} |\mathcal{E}_\lambda(t) - \mathcal{E}_\lambda(u)|
= \max_{\tau \in \Upsilon_\sigma} \mathcal{E}_\lambda^*(\tau) - \min_{\tau \in \Upsilon_\sigma} \mathcal{E}_\lambda^*(\tau)
= W_\lambda(\sigma) \cdot \mathcal{L}(I_\sigma),
\]
as desired.
\end{proof}
}

\section{Hausdorff dimension of the graph of $\mathcal{E}$} \label{sec:graphE}

{
In this section, we prove Theorem~\ref{Hausdorff dimension of graph of E}. We begin with some elementary facts on the fundamental intervals of a fixed order, which will also be used in Section~\ref{sec:graphP}.

\begin{lemma} \label{covering lemma}
Let $n \in \mathbb{N}$. Then the following hold.
\begin{enumerate}[label = \upshape(\roman*), ref = \roman*, leftmargin=*, widest=iii]
\item \label{covering lemma(i)}
The intervals $I_\sigma$, $\sigma \in \Sigma_n$, are pairwise disjoint subsets of $[0,1)$ whose union contains $\mathbb{I}$.
\item \label{covering lemma(ii)}
$\sum_{\sigma \in \Sigma_n} \mathcal{L}(I_\sigma) = 1$.
\item \label{covering lemma(iii)}
$\sup_{\sigma \in \Sigma_n} \mathcal{L}(I_\sigma) \leq 1/(F_{n+1} F_{n+2}) \leq \gamma^{-(2n-1)}$.
\end{enumerate}
\end{lemma}

\begin{proof}
(\ref{covering lemma(i)}) The cylinder sets $\Upsilon_\sigma$, $\sigma \in \Sigma_n$, are pairwise disjoint, and hence so are the intervals $I_\sigma = f^{-1}(\Upsilon_\sigma)$. Given $x \in \mathbb{I}$, we have $\tau \coloneqq f(x) \in \Sigma_\infty$ by Proposition \ref{inverse image of phi}(\ref{inverse image of phi(ii)}), and $\tau^{(n)} \in \Sigma_n$ satisfies $x \in f^{-1}(\Upsilon_{\tau^{(n)}}) = I_{\tau^{(n)}}$.

(\ref{covering lemma(ii)}) Since $\Sigma_n$ is countable and $\mathcal{L}(\mathbb{I}) = 1$, part~(\ref{covering lemma(i)}) and the countable additivity of $\mathcal{L}$ give $\sum_{\sigma \in \Sigma_n} \mathcal{L}(I_\sigma) = \mathcal{L}\big(\bigcup_{\sigma \in \Sigma_n} I_\sigma\big) = 1$.

(\ref{covering lemma(iii)}) By Proposition~\ref{metric properties}(\ref{metric properties(iii)}), we have $q_n^*(\sigma) \geq F_{n+1}$ and $q_n^*(\sigma) + q_{n-1}^*(\sigma) \geq F_{n+1} + F_n = F_{n+2}$ for every $\sigma \in \Sigma_n$. Hence the first inequality follows from \eqref{length of I sigma}, and the second from \eqref{Fibonacci growth}.
\end{proof}
}

Define $G_{\mathcal{E}} \colon \mathbb{R} \to \mathbb{R}^2$ by $G_{\mathcal{E}} (x) \coloneqq (x, \mathcal{E}(x))$ for $x \in \mathbb{R}$. Then the graph of $\mathcal{E}$ is given by $G_{\mathcal{E}} (\mathbb{R}) = \{ ( x, \mathcal{E}(x) ) : x \in \mathbb{R} \}$.

\begin{lemma} \label{Hausdorff dimension of restriction 2}
We have $\hdim G_{\mathcal{E}} ([0,1)) = 1$.
\end{lemma}

\begin{proof}
For the lower bound, consider the projection map onto the first coordinate, $p \colon \mathbb{R}^2 \to \mathbb{R}$, defined by $p (x,y) \coloneqq x$ for each $(x,y) \in \mathbb{R}^2$. Since $p$ is Lipschitz, $\hdim p(F) \leq \hdim F$ for every $F \subseteq \mathbb{R}^2$. Hence,
\[
1 = \hdim [0,1) = \hdim p(G_{\mathcal{E}} ([0,1))) \leq \hdim G_{\mathcal{E}}([0,1)),
\]
as claimed.

For the upper bound, we construct an appropriate covering of $G_{\mathcal{E}}([0,1))$. For each $n \in \mathbb{N}$ and $\sigma \in \Sigma_n$, define the closed interval
\[
J_\sigma \coloneqq
\left[ \inf_{t \in I_\sigma} \mathcal{E}(t), \sup_{t \in I_\sigma} \mathcal{E}(t) \right].
\]
Then $I_\sigma \times \mathcal{E}(I_\sigma) \subseteq I_\sigma \times J_\sigma$. {By Lemma~\ref{covering lemma}(\ref{covering lemma(i)}),} for each $n$, the collection
\[
\mathcal{J}_n \coloneqq \{ I_\sigma \times J_\sigma : \sigma \in \Sigma_n \}
\]
forms a covering of $F \coloneqq G_{\mathcal{E}}(\mathbb{I})$, which differs from $G_{\mathcal{E}}([0,1))$ only by a countable set. By Lemma \ref{sup of difference of E in I sigma}, we have $\mathcal{L}(J_\sigma) = (n+1) \mathcal{L}(I_\sigma)$. Each rectangle $I_\sigma \times J_\sigma$ is thus covered by $n+1$ axis-aligned squares of side length $\mathcal{L}(I_\sigma)$ and diagonal $\sqrt{2} \mathcal{L}(I_\sigma)$. {By Lemma~\ref{covering lemma}(\ref{covering lemma(iii)}), the diameters of these covering squares tend uniformly to zero as $n \to \infty$.}

Let $\varepsilon > 0$. Then the $(1+\varepsilon)$-dimensional Hausdorff measure of $F$ is bounded above as
\begin{align} \label{estimate of Hausdorff measure}
\mathcal{H}^{1+\varepsilon}(F)
\leq \liminf_{n \to \infty} \sum_{\sigma \in \Sigma_n} (n+1) \left( \sqrt{2} \mathcal{L}(I_\sigma) \right)^{1+\varepsilon}
= (\sqrt{2})^{1+\varepsilon} \liminf_{n \to \infty} (n+1) a_n,
\end{align}
where
\begin{align} \label{definition of an}
a_n \coloneqq \sum_{\sigma \in \Sigma_n} (\mathcal{L}(I_\sigma))^{1+\varepsilon} = \sum_{\sigma \in \Sigma_n} \left( \frac{1}{q_n^*(\sigma) (q_n^*(\sigma) + q_{n-1}^*(\sigma))} \right)^{1+\varepsilon}.
\end{align}
{By Lemma~\ref{covering lemma}(\ref{covering lemma(ii)}) and (\ref{covering lemma(iii)}), we have
\[
a_n
= \sum_{\sigma \in \Sigma_n} \mathcal{L}(I_\sigma) \cdot (\mathcal{L}(I_\sigma))^{\varepsilon}
\leq \gamma^{-(2n-1)\varepsilon} \sum_{\sigma \in \Sigma_n} \mathcal{L}(I_\sigma)
= \gamma^{-(2n-1)\varepsilon}.
\]
Hence $(n+1) a_n \leq (n+1) \gamma^{-(2n-1)\varepsilon} \to 0$ as $n \to \infty$, and so \eqref{estimate of Hausdorff measure} yields} $\mathcal{H}^{1+\varepsilon}(F) = 0$. Then $\hdim F \leq 1+\varepsilon$. Since $\varepsilon>0$ was arbitrary, it follows that $\hdim F \leq 1$. Since $F = G_{\mathcal{E}} (\mathbb{I})$ differs from $G_{\mathcal{E}}([0,1))$ only by the countable set $G_{\mathcal{E}}([0,1) \cap \mathbb{Q})$, and since any countable set has Hausdorff dimension $0$, the countable stability property of the Hausdorff dimension (see \cite[pp.~48--49]{Fal14}) implies that $\hdim G_{\mathcal{E}} ([0,1)) = \hdim F \leq 1$. This completes the proof of the lemma.
\end{proof}

\begin{proof} [Proof of Theorem \ref{Hausdorff dimension of graph of E}]
Since $\mathcal{E}$ is periodic with period $1$ by Proposition~\ref{E is periodic}, each set $G_{\mathcal{E}}([n,n+1))$ is a horizontal translation of $G_{\mathcal{E}}([0,1))$. Therefore, by the countable stability property of Hausdorff dimension (see \cite[pp.~48--49]{Fal14}),
\[
\hdim G_{\mathcal{E}}(\mathbb{R}) = \sup_{n\in\mathbb Z} \hdim G_{\mathcal{E}}([n,n+1)) = \hdim G_{\mathcal{E}}([0,1)) = 1,
\]
where the last equality follows from Lemma~\ref{Hausdorff dimension of restriction 2}.
\end{proof}

{
\begin{remark}
We remark that the above proof relies only on the exact oscillation formula in Lemma~\ref{sup of difference of E in I sigma}, the identity $\sum_{\sigma \in \Sigma_n} \mathcal{L}(I_\sigma) = 1$, and the Fibonacci lower bound for the convergent denominators. In Section~\ref{sec:alternative}, we present an alternative argument, which shows that the series $\sum_{n \geq 1} a_n$ converges and evaluates it exactly in terms of the Riemann zeta function and an explicit double series.
\end{remark}
}

\section{An upper bound for the Hausdorff dimension of the graph of $P$} \label{sec:graphP}

{
In this section, we prove Theorem~\ref{Hausdorff dimension of graph of P}, that is, we give a new derivation of the upper bound $3/2$ of Ridley and Petruska \cite{RP00} by using the oscillation formula for $P$ in Lemma~\ref{sup of difference of E in I sigma}. We need the following injectivity property of the convergent denominators.

\begin{lemma} \label{injectivity lemma}
Let $n \in \mathbb{N}$. Then the map $\Sigma_n \to \mathbb{N}^2$ given by $\sigma \mapsto (q_n^*(\sigma), q_{n-1}^*(\sigma))$ is injective. Moreover, if $n \geq 2$, then $1 \leq q_{n-1}^*(\sigma) < q_n^*(\sigma)$ for every $\sigma \in \Sigma_n$.
\end{lemma}

\begin{proof}
If $n = 1$, then $(q_1^*(\sigma), q_0^*(\sigma)) = (\sigma_1, 1)$ determines $\sigma = (\sigma_1)$. Let $n \geq 2$ and $\sigma \in \Sigma_n$. By Proposition~\ref{metric properties}(\ref{metric properties(i)}), we have $q_{n-2}^*(\sigma) \geq q_0^*(\sigma) = 1$, and hence $q_n^*(\sigma) = \sigma_n q_{n-1}^*(\sigma) + q_{n-2}^*(\sigma) > q_{n-1}^*(\sigma) \geq 1$. This proves the second assertion.

Now, suppose that $\sigma, \tau \in \Sigma_n$ satisfy $(q_n^*(\sigma), q_{n-1}^*(\sigma)) = (q_n^*(\tau), q_{n-1}^*(\tau))$, and put $x \coloneqq q_{n-1}^*(\sigma)/q_n^*(\sigma) \in (0,1) \cap \mathbb{Q}$. By the reversal identity in Proposition~\ref{metric properties}(\ref{reversal identity}), we have
\[
\varphi\big( (\sigma_n, \dotsc, \sigma_1) \big) = [\sigma_n, \dotsc, \sigma_1] = x = [\tau_n, \dotsc, \tau_1] = \varphi\big( (\tau_n, \dotsc, \tau_1) \big),
\]
where both reversed sequences $(\sigma_n, \dotsc, \sigma_1)$ and $(\tau_n, \dotsc, \tau_1)$ belong to $\Sigma_n$. By Proposition~\ref{inverse image of phi}(\ref{inverse image of phi(i)}), the fiber $\varphi^{-1}(\{x\})$ consists of exactly two sequences, which belong to $\Sigma_m$ and $\Sigma_{m+1}$ for some $m \in \mathbb{N}$. In other words, the two sequences in $\varphi^{-1}(\{x\})$ have different lengths. Since both reversed sequences lie in $\varphi^{-1}(\{x\}) \cap \Sigma_n$, they must coincide. Hence $\sigma = \tau$, as desired.
\end{proof}
}

Define $G_P \colon \mathbb{R} \to \mathbb{R}^2$ by $G_P (x) \coloneqq (x, P(x))$ for each $x \in \mathbb{R}$, so that $G_P (\mathbb{R}) = \{ ( x, P(x) ) : x \in \mathbb{R} \}$ is the graph of $P$. We now proceed to prove that $\hdim G_P(\mathbb{R}) \leq 3/2$.

\begin{proof}[Proof of Theorem \ref{Hausdorff dimension of graph of P}]
As in the proof of Theorem \ref{Hausdorff dimension of graph of E}, by {Proposition~\ref{E is periodic} and} the countable stability property of the Hausdorff dimension (see \cite[pp.~48--49]{Fal14}), it suffices to show that $\hdim G_{P} ([0,1)) \leq 3/2$.

To establish this upper bound, we construct a suitable cover for $G_P([0,1))$. For each $n \in \mathbb{N}$ and $\sigma \in \Sigma_n$, we define the closed interval
\[
K_\sigma \coloneqq
\left[ \inf_{t \in I_\sigma} P(t), \sup_{t \in I_\sigma} P(t) \right].
\]
{By Lemma~\ref{covering lemma}(\ref{covering lemma(i)}),} for any $n \in \mathbb{N}$, the collection $\mathcal{K}_n \coloneqq \{ I_\sigma \times K_\sigma : \sigma \in \Sigma_n \}$ covers the set $G \coloneqq G_{P}(\mathbb{I})$. By Lemma~\ref{sup of difference of E in I sigma},
\[
\mathcal L(K_\sigma)=\mathcal L(I_\sigma)\sum_{0\leq k\leq n}q_k^*(\sigma).
\]
Therefore, $I_\sigma\times K_\sigma$ can be covered by $\sum_{0\leq k\leq n}q_k^*(\sigma)$ axis-aligned squares of side length $\mathcal L(I_\sigma)$ and diameter $\sqrt{2}\mathcal L(I_\sigma)$. Moreover, by {Lemma~\ref{covering lemma}(\ref{covering lemma(iii)})}, the diameters of these covering squares tend uniformly to zero.

Now, let $\varepsilon > 1/2$ be given. Since $q_0^*(\sigma) \leq q_1^*(\sigma) \leq \dotsb \leq q_n^*(\sigma)$ for any $\sigma \in \Sigma_n$ and $n \in \mathbb{N}$ by Proposition \ref{metric properties}(\ref{metric properties(i)}), it follows that
\begin{align} \label{estimate of Hausdorff measure of G}
\begin{aligned}
\mathcal{H}^{1+\varepsilon}(G)
&\leq \liminf_{n \to \infty} \sum_{\sigma \in \Sigma_n} \left[ (\sqrt{2} \mathcal{L} (I_\sigma))^{1+\varepsilon} \sum_{0 \leq k \leq n} q_k^*(\sigma) \right] \\
&\leq \liminf_{n \to \infty} \sum_{\sigma \in \Sigma_n} \left[ (\sqrt{2} \mathcal{L} (I_\sigma))^{1+\varepsilon} (n+1) q_n^*(\sigma) \right] \\
&= (\sqrt{2})^{1+\varepsilon} \liminf_{n \to \infty} (n+1)b_n,
\end{aligned}
\end{align}
where
\begin{align} \label{definition of bn}
b_n \coloneqq \sum_{\sigma \in \Sigma_n} q_n^*(\sigma) (\mathcal{L}(I_\sigma))^{1+\varepsilon} = \sum_{\sigma \in \Sigma_n} \frac{1}{(q_n^* (\sigma))^\varepsilon (q_n^*(\sigma)+q_{n-1}^*(\sigma))^{1+\varepsilon}}, \quad n \in \mathbb{N}.
\end{align}

{
We now estimate $b_n$ for $n \geq 2$. By Lemma~\ref{injectivity lemma}, distinct sequences $\sigma \in \Sigma_n$ give distinct pairs $(j,k) \coloneqq (q_n^*(\sigma), q_{n-1}^*(\sigma))$ with $1 \leq k < j$. Moreover, $j \geq F_{n+1}$ by Proposition~\ref{metric properties}(\ref{metric properties(iii)}). It follows that
\begin{align*}
b_n
&\leq
\sum_{j\geq F_{n+1}}
\sum_{1\leq k<j}
\frac{1}{j^\varepsilon(j+k)^{1+\varepsilon}}
\leq
\sum_{j\geq F_{n+1}}
\frac{1}{j^\varepsilon}
\cdot\frac{j-1}{(j+1)^{1+\varepsilon}}
\leq
\sum_{j\geq F_{n+1}}\frac{1}{j^{2\varepsilon}} \\
&\leq
\int_{F_{n+1}-1}^{\infty}\frac{dx}{x^{2\varepsilon}}
\leq
\int_{F_n}^{\infty}\frac{dx}{x^{2\varepsilon}}
=
\frac{1}{2\varepsilon-1} \cdot
\frac{1}{F_n^{2\varepsilon-1}},
\end{align*}
where we used $F_{n+1} - 1 \geq F_n \geq 1$ for $n \geq 2$ and $2\varepsilon > 1$. Hence, by \eqref{Fibonacci growth}, we obtain
\[
(n+1)b_n
\leq
\frac{n+1}{2\varepsilon-1} \cdot
\frac{1}{\gamma^{(n-2)(2\varepsilon-1)}}
\to 0
\quad \text{as } n \to \infty.
\]
Substituting this into \eqref{estimate of Hausdorff measure of G},} we deduce that $\mathcal{H}^{1+\varepsilon}(G) = 0$, and therefore $\hdim G \leq 1 + \varepsilon$. Since $\varepsilon > 1/2$ was arbitrary, we conclude that $\hdim G \leq 1+1/2 = 3/2$. Since $G=G_P(\mathbb{I})$ differs from $G_P([0,1))$ only by the countable set $G_P([0,1) \cap \mathbb{Q})$, countable stability of Hausdorff dimension yields
\[
\hdim G_P([0,1))=\max \left\{ \hdim G_P(\mathbb{I}), \hdim G_P([0,1) \cap \mathbb{Q})\right\} \leq \frac{3}{2}.
\]
\end{proof}

\begin{remark}[On the exponent $3/2$] \label{remark on 3/2}
We explain why the covering argument in the above proof cannot yield an upper bound smaller than $3/2$. Let $n \in \mathbb{N}$ and $\sigma \in \Sigma_n$, and put $q_k \coloneqq q_k^*(\sigma)$ and $L \coloneqq \mathcal{L}(I_\sigma)$. By~\eqref{length of I sigma}, we have
\[
q_n L = \frac{1}{q_n + q_{n-1}} = \left( \frac{q_n}{q_n + q_{n-1}} \right)^{1/2} L^{1/2},
\]
where $1/2 \leq q_n/(q_n + q_{n-1}) \leq 1$. Since $q_n \leq W_1(\sigma) \leq (n+1) q_n$, it follows from Lemma~\ref{sup of difference of E in I sigma} that
\begin{align} \label{oscillation of P comparable to sqrt}
2^{-1/2} L^{1/2} \leq \sup_{t,u \in I_\sigma} |P(t) - P(u)| \leq (n+1) L^{1/2}.
\end{align}
On the other hand, the oscillation of $\mathcal{E}$ on $I_\sigma$ equals $(n+1)L$, where $n+1 \leq \frac{1}{2}\log_\gamma(1/L) + \frac{3}{2}$ by Lemma~\ref{covering lemma}(\ref{covering lemma(iii)}). In other words, at the scale of fundamental intervals, $\mathcal{E}$ behaves like a Lipschitz function up to a logarithmic factor, whereas $P$ behaves like a H\"older function of exponent $1/2$. Recall that the graph of a continuous function satisfying a H\"older condition of exponent $\alpha \in (0,1]$ has dimension at most $2 - \alpha$ (see \cite[Chapter~11]{Fal14}), and note that $2 - 1/2 = 3/2$. Since $P$ is not continuous, this is merely a heuristic analogy. Nevertheless, the lower bound in \eqref{oscillation of P comparable to sqrt} shows that the coverings used in the above proof, in which each full column $I_\sigma \times K_\sigma$ with $\sigma \in \Sigma_n$ is covered by squares of side length $\mathcal{L}(I_\sigma)$, cannot yield an upper bound smaller than $3/2$. To see this, let $0 < \varepsilon \leq 1/2$. Since at least $W_1(\sigma) \geq q_n$ squares of side length $L$ are needed to cover the column $I_\sigma \times K_\sigma$, we have
\begin{align*}
\sum_{\sigma \in \Sigma_n} W_1(\sigma) \big(\sqrt{2}\,\mathcal{L}(I_\sigma)\big)^{1+\varepsilon}
&\geq b_n
= \sum_{\sigma \in \Sigma_n} (q_n^*(\sigma) \mathcal{L}(I_\sigma)) (\mathcal{L}(I_\sigma))^{\varepsilon} \\
&\geq 2^{-1/2} \sum_{\sigma \in \Sigma_n} (\mathcal{L}(I_\sigma))^{1/2 + \varepsilon}
\geq 2^{-1/2}
\end{align*}
for every $n \in \mathbb{N}$, where the last inequality follows from $\mathcal{L}(I_\sigma) \leq 1$, $1/2 + \varepsilon \leq 1$, and Lemma~\ref{covering lemma}(\ref{covering lemma(ii)}). Therefore, such coverings can never show that $\mathcal{H}^{1+\varepsilon}(G_P(\mathbb{I})) = 0$ for $\varepsilon \leq 1/2$, no matter how the factor $W_1(\sigma) \leq (n+1) q_n^*(\sigma)$ is estimated. We emphasize that this observation concerns only these particular coverings, and it does not exclude the possibility that other coverings yield a smaller upper bound. For instance, one may use coverings adapted to the distribution of the graph of $P$ within each column $I_\sigma \times K_\sigma$, or coverings whose members meet several columns at once.
\end{remark}

\section{An alternative approach via coprime summation and M\"obius inversion} \label{sec:alternative}

{
In this section, we present an alternative argument for the convergence of the covering sums in Sections~\ref{sec:graphE} and \ref{sec:graphP}. Instead of estimating $a_n$ and $b_n$ for each $n$ separately, we evaluate the series $\sum_{n \geq 1} a_n$ and $\sum_{n \geq 1} b_n$ exactly. To this end, we convert sums over finite CF sequences into sums over coprime pairs of convergent denominators, and then apply M\"obius inversion. Although this argument is not needed for the proofs of Theorems~\ref{Hausdorff dimension of graph of E} and \ref{Hausdorff dimension of graph of P}, it provides exact formulas for these series (Proposition~\ref{exact series}) and shows that the series $\sum_{n \geq 1} b_n$ diverges if and only if $\varepsilon \leq 1/2$.

We begin by recalling two classical identities involving the M\"obius function.} Recall that the Riemann zeta function $\zeta$ is defined by $\zeta (s) \coloneqq \sum_{n \geq 1} 1/n^s$ for each complex number $s$ with $\operatorname{Re}(s) > 1$. The M\"obius function $\mu \colon \mathbb{N} \to \{-1,0,1\}$ is defined for each $n \in \mathbb{N}$ by
\[
\mu (n)
\coloneqq
\begin{cases}
0, &\text{if $n$ has a squared prime factor}; \\
1, &\text{if $n=1$}; \\
(-1)^k, &\text{if $n$ is a product of $k$ distinct primes}.
\end{cases}
\]

\begin{proposition} [See {\cite[Theorem~263]{HW08}}] \label{Mobius identity}
For each $n \in \mathbb{N}$,
\[
\sum_{\substack{d \mid n}} \mu (d) = \begin{cases} 1, &\text{if } n = 1; \\ 0, &\text{if } n > 1. \end{cases}
\]
\end{proposition}

\begin{proposition} [See {\cite[Theorem~287]{HW08}}] \label{Mobius Riemann identity}
For each real number $s>1$,
\[
\frac{1}{\zeta (s)} = \sum_{n \geq 1} \frac{\mu (n)}{n^s}.
\]
\end{proposition}

To {evaluate the series $\sum_{n \geq 1} a_n$ and $\sum_{n \geq 1} b_n$}, we analyze the double sum
\[
S \coloneqq \sum_{n \geq 1} \sum_{\sigma \in \Sigma_n} \Phi \left(q_n^*(\sigma), q_{n-1}^*(\sigma)\right),
\]
where $\Phi$ is a non-negative function on $\mathbb{R}^2$. For every pair of coprime integers $1\leq k<j$, the pair $(j,k)$ occurs exactly twice among
\[
\bigl(q_n^*(\sigma),q_{n-1}^*(\sigma)\bigr), \qquad n\geq1,\ \sigma\in\Sigma_n.
\]
The exceptional pair $(1,1)$ occurs exactly once, whereas the pair $(1,0)$ does not occur. We illustrate this fact with an example below; see also \cite[p.~9]{Els11}.

\begin{example} \label{double sum counting}
Let $q_n^*(\sigma)=10$ for some $\sigma \in \Sigma_n$, and consider $k=3$, which satisfies $\gcd(10,3)=1$ and lies in $\{0,\dotsc,9\}$. Since
\[
\frac{3}{10} = [0; 3,3, \infty, \infty, \dotsc] = [0; 3, 2, 1, \infty, \infty, \dotsc],
\]
the two finite CF representations of $3/10$ are $(3,3)$ and $(3,2,1)$. Reversing these digit sequences, we obtain $(3,3)$ and $(1,2,3)$. By the reversal identity in Proposition~\ref{metric properties}(\ref{reversal identity}), we therefore have
\[
q_2^*((3,3)) = 10,\; q_1^*((3,3)) = 3, \quad \text{and} \quad
q_3^*((1,2,3)) = 10,\; q_2^*((1,2,3)) = 3.
\]
Thus, the $\Phi$-value corresponding to the pair $(10,3)$ appears exactly twice in the double sum.
\end{example}

The following lemma justifies this counting claim.

\begin{lemma}\label{pair correspondence}
The following hold.
\begin{enumerate}[label = \upshape(\roman*), ref = \roman*, leftmargin=*, widest=ii]
\item \label{pair correspondence(i)}
For every pair of coprime integers \(1\le k<j\), there exist some
$n\in\mathbb N$ and exactly two finite CF sequences
\(\sigma\in\Sigma_n\) and \(\sigma'\in\Sigma_{n+1}\) satisfying
\[
(q_n^*(\sigma),q_{n-1}^*(\sigma))=(q_{n+1}^*(\sigma'),q_n^*(\sigma'))=(j,k).
\]
\item \label{pair correspondence(ii)}
There exists exactly one finite CF sequence \(\sigma\) satisfying
\[
(q_n^*(\sigma),q_{n-1}^*(\sigma))=(1,1),
\]
namely \(\sigma=(1)\in\Sigma_1\).
\item \label{pair correspondence(iii)}
There is no finite CF sequence \(\sigma\) satisfying
\[
(q_n^*(\sigma),q_{n-1}^*(\sigma))=(1,0).
\]
\end{enumerate}
{Here, in parts (ii) and (iii), $n \in \mathbb{N}$ denotes the length of $\sigma$, i.e., $\sigma \in \Sigma_n$.}
\end{lemma}

\begin{proof}
(\ref{pair correspondence(i)})
Let $1\le k<j$ with $\gcd(j,k)=1$. Since $k/j\in(0,1)\cap\mathbb Q$, its CF expansion generated by \eqref{CF algorithm 1} and \eqref{CF algorithm 2} is finite. Write
\[
\frac{k}{j}=[b_1,\dotsc,b_{n}],
\quad \text{where } b_{n}>1.
\]
By Proposition~\ref{inverse image of phi}(\ref{inverse image of phi(i)}), the only other finite CF representation of $k/j$ is
\[
[b_1,\dotsc,b_{n}-1,1].
\]
Consider the reversed digit lists
\[
\sigma \coloneqq (b_{n},\dots,b_1) \in \Sigma_{n}, \qquad \sigma' \coloneqq (1,b_{n}-1,b_{{n}-1},\dots,b_1) \in \Sigma_{{n}+1}.
\]
By the reversal identity in Proposition~\ref{metric properties}(\ref{reversal identity}),
\[
\frac{q_{{n}-1}^*(\sigma)}{q_{n}^*(\sigma)} = [b_1,\dotsc,b_{n}] = \frac{k}{j}.
\]
Since the numerator and denominator on each side are relatively prime, $(q_{n}^*(\sigma),\allowbreak q_{{n}-1}^*(\sigma))=(j,k)$. Similarly,
\[
\frac{q_{n}^*(\sigma')}{q_{{n}+1}^*(\sigma')} = [b_1,\dotsc,b_{n}-1,1] = \frac{k}{j},
\]
and hence $(q_{{n}+1}^*(\sigma'),q_{n}^*(\sigma'))=(j,k)$.

For uniqueness, suppose that $\tau \coloneqq (c_1,\dots,c_{\ell}) \in \Sigma_{\ell}$, ${\ell} \in \mathbb{N}$, also satisfies $(q_{\ell}^*(\tau),q_{{\ell}-1}^*(\tau))=(j,k)$. By the reversal identity,
\[
\frac{k}{j} = \frac{q_{{\ell}-1}^*(\tau)}{q_{\ell}^*(\tau)} = [c_{\ell},\dotsc,c_1].
\]
Proposition~\ref{inverse image of phi}(\ref{inverse image of phi(i)}) then implies that either
\[
(c_{\ell},\dots,c_1)=(b_1,\dots,b_{n})
\quad \text{or} \quad
(c_{\ell},\dots,c_1)=(b_1,\dots,b_{n}-1,1).
\]
Reversing the digit lists, we obtain $\tau=\sigma$ or
$\tau=\sigma'$. Thus, there are exactly two such sequences.

(\ref{pair correspondence(ii)})
For $\sigma \coloneqq (1)\in\Sigma_1$, we have $q_1^*(\sigma)=q_0^*(\sigma)=1$.

Conversely, if $\tau \coloneqq (c_1)\in\Sigma_1$ satisfies $(q_1^*(\tau),q_0^*(\tau))=(1,1)$, then $q_1^*(\tau)=c_1=1$, and hence $\tau=(1)$. No sequence of length $n\ge2$ can produce $(1,1)$ because $q_n^*>q_{n-1}^*\ge1$ {by Lemma~\ref{injectivity lemma}}. Therefore, $(1)\in\Sigma_1$ is the unique such sequence.

(\ref{pair correspondence(iii)})
Since $q_{n-1}^*\ge1$ for every $n\ge1$, there is no finite CF sequence $\sigma$ satisfying $(q_n^*(\sigma),q_{n-1}^*(\sigma))=(1,0)$.
\end{proof}

This explains precisely why each coprime pair $(j,k)$ with $1\le k<j$ occurs twice in the double sum, while the exceptional pair $(1,1)$ occurs only once.

The next lemma reformulates the double sum $S$ using the correspondence established in Lemma~\ref{pair correspondence}; it generalizes equation~(9) in \cite[p.~9]{Els11}.

\begin{lemma} \label{summation conversion formula lemma}
Let $\Phi \colon \mathbb{R}^2 \to [0, \infty)$ be a non-negative function. Then{, as an identity in $[0,\infty]$,}
\begin{align} \label{summation conversion formula}
S = -2 \Phi (1,0) + \Phi (1,1) + 2 \sum_{j \geq 1} \sum_{\substack{0 \leq k < j \\ \gcd(j,k)=1}} \Phi (j,k).
\end{align}
\end{lemma}

\begin{proof}
Define
\[
T \coloneqq \sum_{j \geq 1} \sum_{\substack{0 \leq k < j \\ \gcd(j,k)=1}} \Phi(j,k).
\]
By Lemma~\ref{pair correspondence}, each coprime pair $(j,k)$ with $1\leq k<j$ occurs exactly twice among the pairs
\[
\bigl(q_n^*(\sigma),q_{n-1}^*(\sigma)\bigr), \qquad n\geq1,\ \sigma\in\Sigma_n,
\]
whereas the pair $(1,1)$ occurs exactly once and the pair $(1,0)$ does not occur.

On the other hand, $T$ contains each coprime pair $(j,k)$ with $1\leq k<j$ exactly once, contains the pair $(1,0)$, and does not contain the pair $(1,1)$. Therefore, {since all terms are non-negative,}
\[
S = 2T-2\Phi(1,0)+\Phi(1,1),
\]
which is precisely \eqref{summation conversion formula}.
\end{proof}

{
If $\Phi$ satisfies the scaling relation below, the coprimality condition in \eqref{summation conversion formula} can be removed, as the following lemma shows.

\begin{lemma} \label{homogeneous Mobius}
Let $s > 1$, and let $\Phi \colon \mathbb{R}^2 \to [0,\infty)$ satisfy $\Phi(dj, dk) = d^{-s} \Phi(j,k)$ for all $d \in \mathbb{N}$ and all integers $0 \leq k < j$. Then, as an identity in $[0,\infty]$, we have
\begin{align} \label{coprime removal}
\sum_{j \geq 1} \sum_{\substack{0 \leq k < j \\ \gcd(j,k)=1}} \Phi(j,k)
= \frac{1}{\zeta(s)} \sum_{j \geq 1} \sum_{0 \leq k < j} \Phi(j,k).
\end{align}
\end{lemma}
}

\begin{proof}
{Every pair of integers $(j,k)$ with $0 \leq k < j$ can be uniquely written as $(j,k) = (dj', dk')$, where $d \coloneqq \gcd(j,k) \in \mathbb{N}$, $\gcd(j',k') = 1$, and $0 \leq k' < j'$. By non-negativity and the assumed scaling relation, it follows that
\[
\sum_{j \geq 1} \sum_{0 \leq k < j} \Phi(j,k)
= \sum_{d \geq 1} \frac{1}{d^s} \sum_{j' \geq 1} \sum_{\substack{0 \leq k' < j' \\ \gcd(j',k')=1}} \Phi(j',k')
= \zeta(s) \sum_{j \geq 1} \sum_{\substack{0 \leq k < j \\ \gcd(j,k)=1}} \Phi(j,k)
\]
in $[0,\infty]$, which is equivalent to \eqref{coprime removal}.

If the right-hand side of \eqref{coprime removal} is finite, the identity can also be derived directly from M\"obius inversion as follows.} By Proposition \ref{Mobius identity}, for each $j \in \mathbb{N}$,
\begin{align*}
\sum_{\substack{0 \leq k < j \\ \gcd(j,k)=1}} {\Phi(j,k)}
= \sum_{0 \leq k < j} {\Phi(j,k)} \sum_{\substack{d \mid \gcd(j,k)}} \mu(d)
= \sum_{\substack{d \mid j}} \mu(d) \sum_{\substack{0 \leq k < j \\ d \mid k}} {\Phi(j,k)}.
\end{align*}
The resulting triple series is absolutely convergent. Indeed, using non-negativity to interchange the sums and then writing $j=dj'$ and $k=dk'$, we obtain
\begin{align*}
\sum_{j \geq 1}\sum_{\substack{d\mid j}}
\sum_{\substack{0\leq k<j\\d\mid k}}
|\mu(d)| {\Phi(j,k)}
=
\sum_{d\geq1}\frac{|\mu(d)|}{d^{s}}
\sum_{j\geq1}\sum_{0\leq k<j}
{\Phi(j,k)}
<\infty.
\end{align*}
Hence, the order of summation can be interchanged. Rewriting the double sum as a sum over common divisors $d$, {and noting that $0 \leq dk' < dj'$ if and only if $0 \leq k' < j'$,} we obtain
\begin{align*}
\sum_{j \geq 1} \sum_{\substack{0 \leq k < j \\ \gcd(j,k)=1}} {\Phi(j,k)}
&= \sum_{d \geq 1} \mu(d) \sum_{\substack{j \geq 1 \\ d \mid j}} \sum_{\substack{0 \leq k < j \\ d \mid k}} {\Phi(j,k)}
= \sum_{d \geq 1} \mu (d) \sum_{j' \geq 1} \sum_{0 \leq k' < j'} {\Phi(dj', dk')} \\
&= \sum_{d \geq 1} \frac{\mu (d)}{d^{s}} \sum_{j \geq 1} \sum_{0 \leq k < j} {\Phi(j,k)}
= \frac{1}{\zeta ({s})} \sum_{j \geq 1} \sum_{0 \leq k < j} {\Phi(j,k)},
\end{align*}
where the last equality holds by Proposition \ref{Mobius Riemann identity}.
\end{proof}

{
We are now in a position to evaluate the series $\sum_{n \geq 1} a_n$ and $\sum_{n \geq 1} b_n$, where $a_n$ and $b_n$ are defined by~\eqref{definition of an} and \eqref{definition of bn}, respectively. Here, we regard $a_n = a_n(\varepsilon)$ and $b_n = b_n(\varepsilon)$ as defined for every $\varepsilon > 0$.

\begin{proposition} \label{exact series}
Let $\varepsilon > 0$. Then
\begin{align}
\begin{split}
\sum_{n \geq 1} a_n &= -2 + \frac{1}{2^{1+\varepsilon}} + \frac{2}{\zeta (2+2\varepsilon)} \sum_{j \geq 1} \frac{1}{j^{1+\varepsilon}} \sum_{j \leq m < 2j} \frac{1}{m^{1+\varepsilon}} \\
&\leq -2 + \frac{1}{2^{1+\varepsilon}} + \frac{2 \zeta(1+2\varepsilon)}{\zeta (2+2\varepsilon)} < \infty,
\end{split} \label{series of an} \\
\sum_{n \geq 1} b_n &= -2 + \frac{1}{2^{1+\varepsilon}} + \frac{2}{\zeta (1+2\varepsilon)} \sum_{j \geq 1} \frac{1}{j^{\varepsilon}} \sum_{j \leq m < 2j} \frac{1}{m^{1+\varepsilon}}, \label{series of bn}
\end{align}
where \eqref{series of bn} is an identity in $[0,\infty]$. Moreover, $\sum_{n \geq 1} b_n < \infty$ if and only if $\varepsilon > 1/2$, in which case $\sum_{n \geq 1} b_n \leq -2 + 2^{-1-\varepsilon} + 2\zeta(2\varepsilon)/\zeta(1+2\varepsilon)$.
\end{proposition}

\begin{proof}
By \eqref{length of I sigma}, we have $a_n = \sum_{\sigma \in \Sigma_n} \Phi_a(q_n^*(\sigma),\allowbreak q_{n-1}^*(\sigma))$ and $b_n = \sum_{\sigma \in \Sigma_n} \Phi_b(q_n^*(\sigma),\allowbreak q_{n-1}^*(\sigma))$, where
\[
\Phi_a(j,k) \coloneqq \frac{1}{j^{1+\varepsilon}(j+k)^{1+\varepsilon}}
\quad \text{and} \quad
\Phi_b(j,k) \coloneqq \frac{1}{j^{\varepsilon}(j+k)^{1+\varepsilon}}
\]
for $j \geq 1$ and $k \geq 0$, and $\Phi_a \coloneqq \Phi_b \coloneqq 0$ elsewhere. Notice that both functions take the value $1$ at $(1,0)$ and the value $2^{-1-\varepsilon}$ at $(1,1)$. They satisfy the scaling condition in Lemma~\ref{homogeneous Mobius} for all integers $0\leq k<j$ and all $d\in\mathbb{N}$, with $s=2+2\varepsilon$ and $s=1+2\varepsilon$, respectively. Moreover, for each $j \in \mathbb{N}$, we have
\[
\sum_{0 \leq k < j} \Phi_a(j,k) = \frac{1}{j^{1+\varepsilon}} \sum_{j \leq m < 2j} \frac{1}{m^{1+\varepsilon}}
\quad \text{and} \quad
\sum_{0 \leq k < j} \Phi_b(j,k) = \frac{1}{j^{\varepsilon}} \sum_{j \leq m < 2j} \frac{1}{m^{1+\varepsilon}}.
\]
Applying Lemma~\ref{summation conversion formula lemma} and then Lemma~\ref{homogeneous Mobius} with $s = 2+2\varepsilon$ and $s = 1+2\varepsilon$, respectively, we obtain the identities in \eqref{series of an} and \eqref{series of bn}. Now, since
\begin{align} \label{inner sum bounds}
\frac{1}{2^{1+\varepsilon}} \cdot \frac{1}{j^{\varepsilon}}
= j \cdot \frac{1}{(2j)^{1+\varepsilon}}
\leq \sum_{j \leq m < 2j} \frac{1}{m^{1+\varepsilon}}
\leq j \cdot \frac{1}{j^{1+\varepsilon}}
= \frac{1}{j^{\varepsilon}}
\end{align}
for each $j \in \mathbb{N}$, the double sum in \eqref{series of an} is at most $\zeta(1+2\varepsilon) < \infty$, and the double sum in \eqref{series of bn} lies between $2^{-1-\varepsilon} \sum_{j \geq 1} j^{-2\varepsilon}$ and $\sum_{j \geq 1} j^{-2\varepsilon}$. Since the latter series converges if and only if $2\varepsilon > 1$, in which case its sum is $\zeta(2\varepsilon)$, the proof is complete.
\end{proof}

Proposition~\ref{exact series} provides alternative proofs of Theorems~\ref{Hausdorff dimension of graph of E} and \ref{Hausdorff dimension of graph of P}. Indeed, since $\sum_{n \geq 1} a_n$ converges for every $\varepsilon > 0$, we have $\liminf_{n \to \infty} (n+1) a_n = 0$. To see this, suppose for contradiction that $\liminf_{n \to \infty} (n+1) a_n > 0$. Then there exist $c > 0$ and $N \in \mathbb{N}$ such that $a_n > c/(n+1)$ for all $n \geq N$, which contradicts the convergence of $\sum_{n \geq 1} a_n$. Substituting $\liminf_{n \to \infty} (n+1) a_n = 0$ into \eqref{estimate of Hausdorff measure}, we obtain $\mathcal{H}^{1+\varepsilon}(G_{\mathcal{E}}(\mathbb{I})) = 0$ for every $\varepsilon > 0$, which yields Lemma~\ref{Hausdorff dimension of restriction 2} and hence Theorem~\ref{Hausdorff dimension of graph of E}. In the same way, the convergence of $\sum_{n \geq 1} b_n$ for $\varepsilon > 1/2$, together with \eqref{estimate of Hausdorff measure of G}, yields Theorem~\ref{Hausdorff dimension of graph of P}.

\begin{remark}
If $\varepsilon \leq 1/2$, then the series $\sum_{n \geq 1} b_n$ diverges by Proposition~\ref{exact series}, and so the above argument cannot yield an upper bound smaller than $3/2$ for the Hausdorff dimension of the graph of $P$. We note that the divergence of the series alone does not exclude the possibility that $\liminf_{n \to \infty} (n+1) b_n = 0$. However, Remark~\ref{remark on 3/2} shows the stronger statement that $b_n \geq 2^{-1/2}$ for every $n \in \mathbb{N}$ when $\varepsilon \leq 1/2$. Hence, both the argument in this section and that in Section~\ref{sec:graphP} fail for every $0 < \varepsilon \leq 1/2$ for the same reason, namely, the lower bound \eqref{oscillation of P comparable to sqrt} for the oscillation of $P$.
\end{remark}
}

\section{Conclusions} \label{sec:conclusions}

{
In this paper, we have shown that the graph of the unweighted error-sum function
\[
\mathcal{E}(x)
=
\sum_{n\geq0}
\left(
x-\frac{p_n(x)}{q_n(x)}
\right)
\]
of continued fractions has Hausdorff dimension $1$ (Theorem~\ref{Hausdorff dimension of graph of E}). The key step is the exact oscillation formula
\[
\sup_{t,u \in I_\sigma} |\mathcal{E}(t) - \mathcal{E}(u)| = (n+1) \cdot \mathcal{L}(I_\sigma), \qquad \sigma \in \Sigma_n,
\]
in Lemma~\ref{sup of difference of E in I sigma}. We obtained this formula by lifting $\mathcal{E}$ to the symbolic space $\Sigma$ of CF sequences, which contains the non-realizable finite sequences representing the missing endpoints of fundamental intervals, and by analyzing the alternating tails of the lifted series. Once the formula is established, Theorem~\ref{Hausdorff dimension of graph of E} follows from a short covering argument based on the Fibonacci lower bound for the convergent denominators.

As a consistency check of our framework, we applied the same method to the relative error-sum function $P$, whose oscillation on $I_\sigma$ equals $\mathcal{L}(I_\sigma) \sum_{k=0}^n q_k^*(\sigma)$, and recovered the upper bound $3/2$ of Ridley and Petruska \cite{RP00} for the Hausdorff dimension of the graph of $P$ (Theorem~\ref{Hausdorff dimension of graph of P}). The alternative approach in Section~\ref{sec:alternative}, which is based on the correspondence between finite CF sequences and coprime pairs of convergent denominators together with M\"obius inversion, evaluates the relevant covering series exactly.

Our analysis also clarifies the limitation of the method. By Remark~\ref{remark on 3/2}, the oscillation of $P$ on every fundamental interval $I_\sigma$ is at least $2^{-1/2} \mathcal{L}(I_\sigma)^{1/2}$. Consequently, the coverings used in the proof of Theorem~\ref{Hausdorff dimension of graph of P}, in which each full column over a fundamental interval of order $n$ is covered by squares of side length $\mathcal{L}(I_\sigma)$, cannot yield an upper bound smaller than $3/2$, no matter how the factor $\sum_{k=0}^n q_k^*(\sigma)$ is estimated. In accordance with this, the covering series $\sum_{n \geq 1} b_n$ diverges if and only if $\varepsilon \leq 1/2$ (Proposition~\ref{exact series}). Figure~\ref{fig:graphs} suggests that the graph of $P$ is much more irregular than that of $\mathcal{E}$, but we do not have enough evidence to conjecture the exact value of its dimension. Therefore, the exact Hausdorff dimension of $G_P(\mathbb{R})$ remains unknown. In particular, it would be interesting to determine whether the upper bound $3/2$ is sharp. Any improvement through the present covering approach would require different coverings, for instance, coverings adapted to the internal structure of the graph of $P$ within each column. Of course, other kinds of arguments for the dimension of the graph are not ruled out by our analysis.
}

\section*{Acknowledgements}

The present research was supported by the research fund of Dankook University in 2026.

\end{document}